\documentclass[journal]{IEEEtran}
\pdfoutput=1
\usepackage{amsmath,amsfonts,amssymb,amsthm}
\usepackage{algorithmic}
\usepackage{algorithm}
\usepackage{array}
\usepackage[caption=false,font=normalsize,labelfont=sf,textfont=sf]{subfig}
\usepackage{textcomp}
\usepackage{stfloats}
\usepackage{url}
\usepackage{verbatim}
\usepackage{graphicx}
\usepackage{cite}
\usepackage{multirow}
\usepackage{color}
\usepackage{threeparttable}
\usepackage{extarrows}

\newcommand{\red}[1]{\textcolor[rgb]{1.00,0.00,0.00}{#1}}
\newtheorem{theorem}{Theorem}

\newtheorem{lemma}{Lemma}
\newtheorem{assumption}{Assumption}

\newtheorem{remark}{Remark}

\allowdisplaybreaks

\begin{document}

\title{Distributed Momentum-based Frank-Wolfe Algorithm for Stochastic Optimization}

\author{
	Jie Hou, Xianlin Zeng,~\IEEEmembership{Member, IEEE}, Gang Wang,~\IEEEmembership{Member, IEEE}, \\Jian Sun,~\IEEEmembership{Senior Member, IEEE}, and Jie Chen,~\IEEEmembership{Fellow, IEEE}
\thanks{This work was supported in part by the National Key R$\&$D Program of China under Grant 2021YFB1714800,  the National Natural Science Foundation of China under Grants  62073035, 62173034, 61925303, 62088101, 61873033, the CAAI-Huawei MindSpore Open Fund, and the Chongqing Natural Science Foundation under Grant 2021ZX4100027. \textit{(Corresponding author: Xianlin Zeng.)}

J. Hou and X. Zeng are with the Key Laboratory of Intelligent Control and Decision of Complex Systems, School of Automation, Beijing Institute of Technology, Beijing, 100081, China (E-mail: houjie@bit.edu.cn; xianlin.zeng@bit.edu.cn).

G. Wang and J. Sun are with the Key Laboratory of Intelligent Control and Decision of Complex Systems, School of Automation, Beijing Institute of Technology, Beijing, 100081, China and Beijing Institute of Technology Chongqing Innovation Center, Chongqing, 401120, China (E-mail: gangwang@bit.edu.cn; sunjian@bit.edu.cn).

J. Chen is with the School of Electronic and Information Engineering, Tongji University, Shanghai, 200082, China and also with the Key Laboratory of Intelligent Control and Decision of Complex Systems, School of Automation, Beijing Institute of Technology, Beijing, 100081, China (E-mail: chenjie@bit.edu.cn).}
}

\markboth{Accepted to IEEE/CAA Journal of Automatica Sinica on \today}%
{Distributed Momentum-based Frank-Wolfe Algorithm for Stochastic Optimization}


\maketitle

\begin{abstract}
This paper considers distributed stochastic optimization, in which a number of agents cooperate to optimize a global objective function through local computations and information exchanges with neighbors over a network.  Stochastic optimization problems are usually tackled by variants of projected stochastic gradient descent. However, projecting a point onto a feasible set is often expensive. The Frank-Wolfe (FW) method has well-documented merits in handling convex constraints, but existing stochastic FW algorithms are basically developed for centralized settings. In this context, the present work puts forth a distributed stochastic Frank-Wolfe solver, by judiciously combining Nesterov's momentum and gradient tracking techniques for stochastic convex and nonconvex optimization over networks. It is shown that the convergence rate of the proposed algorithm is $\mathcal{O}(k^{-\frac{1}{2}})$ for convex optimization, and $\mathcal{O}(1/\mathrm{log}_2(k))$ for nonconvex optimization. The efficacy of the algorithm is demonstrated by numerical simulations against a number of competing alternatives.
\end{abstract}

\begin{IEEEkeywords}
Distributed Optimization, Frank-Wolfe Algorithms, Stochastic Optimization, Momentum-based Method.
\end{IEEEkeywords}

\section{Introduction}
\IEEEPARstart{D}{istributed} stochastic optimization is a basic problem that arises widely in  diverse engineering applications, including unmanned systems \cite{eng2022,Jiang_SC_2022,YangOver2022},   
distributed machine learning \cite{Springer2010LBottou}, and multi-agent reinforcement learning \cite{sun2020marvel,wang2020decentralized,WangDDistributed2022}, to name a few. The goal is to minimize a shared objective function, which is defined as the expectation of a set of stochastic functions subject to general convex constraints, by means of local computations and information exchanges between working agents.

This paper considers a set $\mathcal{N}=\{1,2,\cdots,n\}$ of working agents  connected through a communication network $\mathcal{G}=(\mathcal{N},\mathcal{E})$, where $\mathcal{E}\subseteq \mathcal{N}\times\mathcal{N}$ denotes the set of edges. Each agent $i\in\mathcal{N}$ has a local  objective function $F_i(x)=\mathbb{E}[f_i(x,\xi^i)]$, where  $f_i$ is a stochastic function involving strategy variable $x\in\mathbb{R}^p$ and random variable $\xi^i$ that follows an unknown distribution. The collective goal of all the agents is to find $x^\ast$ that minimizes the average of all objective functions, i.e.,
\begin{equation}\label{ss1}
	\mathop{\mathrm{min}}\limits_{x\in\mathcal{X}}~F(x):=\frac{1}{n}\sum_{i=1}^nF_i(x)
\end{equation}
where $\mathcal{X}\subset\mathbb{R}^p$ is a convex and compact feasible set. Problems of the form \eqref{ss1} lie at the heart of machine learning and adaptive filtering, emerging in e.g., clustering, classification, energy management, and resource allocation \cite{jmlr2020AMokhtari,acc2021ZAkhtar,Allerton2016SJReddi,GongUAV2022}.

 A popular approach to solving problem \eqref{ss1} is the projected stochastic gradient descent (pSGD) \cite{tsp2020SPU,tsp2021Zwang,wang2019learning}. In pSGD and its variants, the iteration variable is projected back onto $\mathcal{X}$ after taking a step in the direction of negative stochastic gradient \cite{DaiSyn2020,sjo2009ANemirovski,twc2015Yzhang,tie2019Xwang,sjo2013SGhadimi,JAS-2020-1315}. Such algorithms are efficient when the computational cost of performing the projection is low, {e.g., projecting onto a hypercube or simplex.} In many practical situations of interest, however, the cost of projecting onto $\mathcal{X}$ can be high, e.g., dealing with a trace norm ball or a base polytope $\mathcal{X}$  in submodular minimization  \cite{pjo2011SFujishige}. 

An alternative for tackling problem \eqref{ss1} is the projection-free methods, including the Frank-Wolfe (FW) \cite{nrl1956Frank} and conditional gradient sliding \cite{sjo2016GLan}.
In this paper, we focus on the FW algorithm, which is also known as conditional gradient method \cite{nrl1956Frank}. Classical FW methods circumvent the projection step by first solving a linear minimization subproblem over the constraint set $\mathcal{X}$ to obtain a sort of conditional gradient $\theta_k$, which is followed by updating $x_{k+1}$ through a convex combination of the current iteration variables $x_k$ and $\theta_k$. 
On top of this idea, a number of modifications have been proposed to improve or accelerate the FW method in algorithm design or convergence analysis, see e.g., \cite{tsp2021DSKalhan,picais2021TKerdreux,tsp2021BLi,ICASSP2021YZhang,pol2012EHazan,icml2016EHazan,JAS-2020-0016,jmlr2020AMokhtari,acc2021ZAkhtar,Allerton2016SJReddi,ijna2021MWeber}. 

Nonetheless, most existing projected stochastic gradient descent (pSGD) and  Frank-Wolfe variants are designed for constrained centralized problems, and they cannot directly handle distributed problems.
Therefore, it is necessary to develop distributed stochastic projection-free methods for problem \eqref{ss1}.  In addition, stochastic FW methods may not converge even in the centralized convex case,  without increasing the batch size \cite{jmlr2020AMokhtari}. In this context, a natural question arises: \textit{is it possible to develop a distributed FW method by using any fixed batch size for problem \eqref{ss1}, while enjoying a convergence rate comparable to that of centralized stochastic FW methods?}
In this paper, we answer this question affirmatively, by carefully designing a distributed stochastic FW algorithm, which converges for any fixed batch size (can be as small as 1) and enjoys a comparable convergence rate as in the centralized stochastic setting.

\IEEEpubidadjcol

 \subsection{Related Works}

  \textbf{Projection-free stochastic} algorithms for addressing stochastic optimization problems were widely studied in recent years. When both the global function $F$ and the feasibility set $\mathcal{X}$ in \eqref{ss1} are convex,  an online stochastic FW method using minibatches was proposed and shown to converge at a rate of  $\mathcal{O}(k^{-\frac{1}{4}})$  \cite{pol2012EHazan}. By progressively increasing the batch size per iteration, convergence rate of stochastic FW algorithms was improved to $\mathcal{O}(k^{-\frac{1}{3}})$ in \cite{icml2016EHazan}. The work \cite{jmlr2020AMokhtari} further relaxed the requirement of increasing the batchsize by using a fixed small batch size along with some heuristics, while maintaining the convergence rate of  $\mathcal{O}(k^{-\frac{1}{3}})$. Faster rate $\mathcal{O}(k^{-\frac{1}{2}})$ was obtained by merging the Nesterov's momentum and classical FW method in \cite{acc2021ZAkhtar}.  When $F$ becomes nonconvex, \cite{Allerton2016SJReddi} proposed a stochastic FW variant, and established a convergence rate of $\mathcal{O}(k^{-\frac{1}{4}})$ for handling nonconvex stochastic optimization problems. Lately, \cite{ijna2021MWeber} studied nonconvex stochastic optimization on Riemannian manifolds and presented a projection-free stochastic algorithm which achieves the same convergence rate as in \cite{Allerton2016SJReddi}.
 It is worth noting that the aforementioned FW methods are all centralized. Thus far, distributed projection-free stochastic algorithms have rarely been studied.

\textbf{Distributed FW} methods play an important role in distributed convex and nonconvex optimization, a sample of which can be found in \cite{SDM2015ABellet,GlobalSIP2016HWai,tac2017HWai,arxi2021GChen,tsp2017LZhang,AISTATS2019JXie,IJCAI2021HGao}. In the deterministic case, a distributed FW algorithm was developed in \cite{GlobalSIP2016HWai} for a class of nonconvex optimization problems. The work \cite{tac2017HWai} further devised distributed FW algorithms, and showed convergence rate of $\mathcal{O}(k^{-1})$ for convex optimization and $\mathcal{O}(k^{-\frac{1}{2}})$ for nonconvex optimization. For submodular maximization, \cite{AISTATS2019JXie} proposed two distributed algorithms for deterministic and stochastic optimization, and obtained the convergence rate of $\mathcal{O}(k^{-\frac{1}{3}})$ for stochastic optimization. The convergence rate in \cite{AISTATS2019JXie} was improved  to $\mathcal{O}(k^{-\frac{1}{2}})$ \cite{IJCAI2021HGao} by using variance reduction techniques and gradient tracking strategies.

Although considerable results have been reported for distributed FW in deterministic settings, they cannot be directly applied in and/or generalized to stochastic settings. The reason is twofold: i) FW may diverge due to the non-vanishing variance in gradient estimates; and, ii) the desired convergence rate of FW for
 stochastic optimization is not guaranteed to be comparable to pSGD, even for the centralized setting.

To address these challenges, the present paper puts forth a distributed stochastic version of the celebrated FW algorithm for stochastic optimization over networks. The main idea behind our proposal is a judicious combination of the recursive momentum \cite{panips2019ACutkosky} and the Nesterov's momentum \cite{dans1983Nesterov}. On the theory side, it is shown that the proposed algorithm can not only attenuate the noise in gradient approximation, but also achieve a convergence guarantee comparable to pSGD in convex case. Comparison of the proposed algorithm in context is provided in Table \ref{table1}.


\begin{table}[!]
\begin{center}
\caption{Convergence Rate for Stochastic Optimization}\label{table1}
\begin{tabular}{m{1.25cm}<{\centering}|c|m{1.8cm}<{\centering}|m{1.38cm}<{\centering}|c}
\hline
Reference & Setting & Projection-free & Function & Rate\\\hline
RSA \cite{sjo2009ANemirovski}&centralized &unconstrained & smooth \red{convex} & $\mathcal{O}(k^{-\frac{1}{2}})$\\\hline
RSG \cite{sjo2013SGhadimi}&centralized&no& smooth \red{nonconvex}& $\mathcal{O}(k^{-\frac{1}{2}})$\\\hline
SPPDM \cite{tsp2021Zwang}&\red{distributed}&unconstrained& nonsmooth \red{nonconvex}& $\mathcal{O}(k^{-\frac{1}{2}})$\\\hline
OFW \cite{pol2012EHazan}& centralized& yes&smooth\newline $Q$-Lipschitz \red{convex} & $\mathcal{O}(k^{-\frac{1}{4}})$\\\hline
SFW \cite{jmlr2020AMokhtari}& centralized& yes&smooth \red{convex}& $\mathcal{O}(k^{-\frac{1}{3}})$\\\hline
MSHFW \cite{acc2021ZAkhtar}&centralized&yes&smooth \red{convex}& $\mathcal{O}(k^{-\frac{1}{2}})$\\\hline
NSFW \cite{Allerton2016SJReddi}&centralized&yes&smooth  $Q$-Lipschitz \red{nonconvex}& $\mathcal{O}(k^{-\frac{1}{4}})$\\\hline
SRFW \cite{ijna2021MWeber}&centralized&yes&smooth  $Q$-Lipschitz \red{nonconvex}& $\mathcal{O}(k^{-\frac{1}{4}})$\\\hline
\multirow{2}{*}{\textbf{This Work}}&\multirow{2}{*}{\red{distributed}}&\multirow{2}{*}{yes}&smooth \red{convex}&$\mathcal{O}(k^{-\frac{1}{2}})$\\\cline{4-5}
 & & & smooth \red{nonconvex}&$\mathcal{O}(\frac{1}{\mathrm{log}_2(k)})$\\
\hline
\end{tabular}
\begin{tablenotes}
	\item[1] *  The function $f_i(x,\xi^i)$ in (\ref{ss1}) is $Q$-Lipschitz if $\|\nabla f_i(x,\xi^i)\|\leqslant Q$ for all $\xi^i$, where $Q$ is a positive constant \cite{Allerton2016SJReddi}; $k$ denotes the number of iterations.
\end{tablenotes}
\end{center}
\end{table}

 \subsection{Our contributions}
In succinct form, the contributions of this work are summarized as follows.
\begin{enumerate}
\item 	We propose a projection-free algorithm, referred to as the distributed momentum-based Frank-Wolfe (DMFW), for convex and nonconvex stochastic optimization over networks. Compared with the centralized FW methods \cite{pol2012EHazan,jmlr2020AMokhtari,acc2021ZAkhtar,Allerton2016SJReddi,ijna2021MWeber}, DMFW is considerably different in algorithm design and convergence analysis.

\item For convex objective functions, we establish a  convergence rate of $\mathcal{O}(k^{-\frac{1}{2}})$ for DMFW, which matches that of distributed pSGD \cite{sjo2009ANemirovski} and is even faster than those of centralized FW algorithms in \cite{pol2012EHazan,jmlr2020AMokhtari}.
\item For nonconvex objective functions,  we establish a convergence rate of $\mathcal{O}(1/\mathrm{log}_2(k))$ for DMFW, 
which, to the authors' best knowledge, marks the first FW's convergence rate result for distributed nonconvex stochastic optimization.

\end{enumerate}


\section{Preliminaries and Algorithm Design}
\subsection{Notation and preliminaries}

Let $\mathbb{R}$ denote the set of real numbers, and $\mathbb{R}^p$ the set of $p$-dimensional real vectors;
$\langle\cdot\rangle$ denotes the inner product; $(\cdot)^\mathrm{T}$ represents the transpose;
 $\|x\|$ denotes the $l_2$ norm (Euclidean norm) of vector $x$, and $\|x\|_q$ $(q\in[1,+\infty))$ symbols the $l_q$ norm of vector $x$; $\mathrm{max}\{\cdot\}$ denotes the maximum element in set $\{\cdot\}$; {$\lceil\cdot\rceil$ is the ceiling operation}; $\mathbb{E}[\cdot]$ is the expectation operator; $\mathbb{E}_k[\cdot]$ is the conditional expectation $\mathbb{E}[\cdot|\mathcal{F}_k]$ on the sigma field $\mathcal{F}_k$ which contains all types of randomness up to iteration $k$;  $C=[c_{ij}]_{n\times n}$ is the weighted adjacency matrix of graph $\mathcal{G}(\mathcal{N},\mathcal{E})$. For $\forall i,j\in\mathcal{N}$, if $(i,j)\in\mathcal{E}$, then $c_{ij}>0$, and $c_{ij}=0$ otherwise.

 Consider a differentiable function $F:\mathbb{R}^p\rightarrow \mathbb{R}$, whose gradient is $\nabla F(\cdot)$. The function is $L$-smooth over a convex set $\mathcal X$ if
$$F(x)-F(y)\leqslant \langle\nabla F(y),x-y\rangle+\frac{L}{2} \|x-y\|^2,~~\forall x,y\in\mathcal{X}$$
where $L>0$ is a constant.
The function $F$ is said to be convex over a convex set $\mathcal X$ if
$F(x)-F(y)\geqslant \langle\nabla F(y),x-y\rangle$ for all $x,y\in\mathcal{X}.$

\subsection{Algorithm design}
To solve problem \eqref{ss1}, we propose a distributed momentum-based Frank-Wolfe algorithm, which is summarized in Algorithm \ref{alg:1}.
\begin{algorithm}
	\renewcommand{\algorithmicrequire}{\textbf{Input:}}
	\renewcommand{\algorithmicensure}{\textbf{Output:}}
	\caption{Distributed Momentum-based Frank-Wolfe}
	\label{alg:1}
	\begin{algorithmic}[1]
		\REQUIRE number of iterations $K$, initial condition $x^i_{1}\in\mathcal{X}$, and $y^i_{1}=\nabla f_i(\hat{x}_{1}^i,\xi_{1}^i)=s_{1}^i$ for $\forall i\in\mathcal{N}$.
		\FORALL {$k=1,2,\cdots,K$}
		\STATE \textit{Average consensus:} 
		\begin{align}\label{alg1step1}
			\hat{x}^i_{k}=\sum\limits_{j\in\mathcal{N}_i}c_{ij}x^j_{k}
       \end{align}
		where $\mathcal{N}_i$ is the set of neighbors of node $i$.
		\STATE \textit{Momentum update:} 
         \begin{align}\label{algstep3}
          y^i_{k}=&(1-\gamma_k)y^{i}_{k-1}+\nabla f_i(\hat{x}^i_{k},\xi^i_{k})\nonumber\\
          &-(1-\gamma_k)\nabla f_i(\hat{x}^{i}_{k-1},\xi^i_{k})
         \end{align}
          where $\gamma_k\in(0,1]$ is a step size.
        \STATE \textit{Gradient tracking:} 
        \begin{align}\label{algeua1}
          s^i_{k}&= \sum\limits_{j\in\mathcal{N}_i}c_{ij}s^{j}_{k-1}+y^i_{k}-y^{i}_{k-1}\\\label{algeua2}
           p^i_{k}&=\sum\limits_{j\in\mathcal{N}_i}c_{ij}s^j_{k}
        \end{align}
		\STATE \textit{Frank-Wolfe step:} 
		 \begin{align}\label{algstep4}
		 	&\theta_k^i\in\mathop{\mathrm{argmin}}\limits_{\phi\in\mathcal{X}}\langle p^i_k,\phi\rangle\\\label{algstep5}
		 		 	&x^i_{k+1}=\hat{x}^i_k+\eta_k(\theta^i_k-\hat{x}^i_k)
		 \end{align}
		where $\eta_k\in(0,1]$ is a step size.
		\ENDFOR
		\STATE \textbf{return} $x_{k+1}^i$ for all $i\in\mathcal{N}$.
		\end{algorithmic}
\end{algorithm}

\textit{Average consensus:}
We employ the average consensus (AC) protocol \cite{John1984Problems,ZLH_JAS_2014,RLXS_JAS_2021,Juan_SC_2022}, in which an agent  takes a weighted average of the values from its neighbors according to $C$.

\textit{Momentum update:} Because the distribution of $\xi^i$ in \eqref{ss1} is unknown, we can only have access to stochastic gradients of $F_i(x)$, that is, for a given $x\in\mathcal{X}$ and randomly sampled $\xi^i$, the oracle returns $\nabla f_i(x,\xi^i)$, which is assumed to be an unbiased estimate of $\nabla F_i(x)$. It is well known that the naive stochastic implementation of Frank-Wolfe by replacing $\nabla F_i(x)$ with $\nabla f_i(x,\xi^i)$, may diverge due to the non-vanishing variance of $\nabla f_i(x,\xi^i)$. To address this issue, we generalize the recursive momentum in \cite{panips2019ACutkosky} to distributed stochastic optimization.

\textit{Gradient tracking:}  Inspired by the gradient tracking method in \cite{tsin2016PDLorenzo,TNNLS2021Li},  which reuses the global gradient $p_{k-1}^i$ from the last iteration, agent $i$ at iteration $k$ approximates the global gradient via \eqref{algeua1} and \eqref{algeua2}. The initialization $s^i_{1}=\nabla f_i(\hat{x}_{1}^i,\xi_{1}^i)$ is set for $\forall i\in\mathcal{N}$.

\textit{Frank-Wolfe step:} A feasible direction $\theta_k^i$ is obtained by minimizing its correlation with $p_i(k)$ over $\mathcal{X}$ in \eqref{algstep4}. Subsequently,  the variable $x^i_{k+1}$ is generated as a convex combination of $\hat{x}^i_{k}$ and $\theta_k^i$.


\begin{remark}
There are two mechanisms for information exchanging with neighbors in DMFW: (i)  \textbf{average consensus}; and (ii) \textbf{gradient tracking}. 

In the \textbf{average consensus} step, agent $i$ approximates the average iteration by exchanging the latest iteration information with its neighbors.  In the \textbf{gradient tracking} step, agent $i$ approximates the global gradient by weighted averaging  $s^j_{k-1}$ and $y_k^i- y_{k-1}^i$, which is an estimate of local gradient difference.

\end{remark}

\begin{remark}
Compared with the existing distributed solutions \cite{tsin2016PDLorenzo,TNNLS2021Li,tac2017HWai,JSSC2021Distributed}, Algorithm \ref{alg:1} shares very similar structures: \textbf{global  consensus} steps plus \textbf{local adaptation} steps.

\textbf{global consensus}: Algorithm \ref{alg:1} realizes the distributed update by exploiting a twofold consensus-based
mechanism to: (i) enforce an agreement among the agents' estimates $\hat{x}^i_k$; and (ii) dynamically track the gradient of the whole cost function through an auxiliary variable $s^i_k$. 

\textbf{local adaptation}: Agent $i$ approximates the local gradient $y^i_k$ and updates its variable $x^i_{k+1}$ independently via local learning process by using $\hat{x}^i_k$ and $s^i_k$ obtained from the \textbf{global consensus} steps. 
\end{remark}

 \begin{remark}
 It is worth mentioning that Algorithm \ref{alg:1} works with a single stochastic gradient (i.e., with batch size as small as $1$), unlike the methods in \cite{AISTATS2019JXie},  which requires increasing the batch sizes as the number of iterations $k$ grows.
\end{remark}

\section{Main Results}
In this section, we establish the convergence results of the proposed algorithm for convex and nonconvex problems, respectively. Before providing the results, we outline some standing assumptions and facts.
\subsection{Assumptions and facts}
\begin{assumption}[Weight rule]\label{ass1}
	The weighted adjacency matrix $C$ is a doubly stochastic matrix, i.e.,  the row sum and the column sum of $C$ are all $1$.
	\end{assumption}
	
	Assumption \ref{ass1} indicates that for each round of the \textit{Average Consensus} step of Algorithm \ref{alg:1},  the agent takes a weighted average of the values from its neighbors according to $C$.
	
	\begin{assumption}[Connectivity]\label{asssss2}
	The network $\mathcal{G}$ is connected. 
	\end{assumption}
	 
	 If Assumptions \ref{ass1} and \ref{asssss2} hold, the magnitude of the second largest eigenvalue  of the weighted adjacency matrix $C$, denoted by $\lambda$, is strictly less than one, i.e., $|\lambda|<1$ \cite{tac2017HWai}.
	 The following fact holds for any doubly stochastic matrix $C$.

	\textbf{Fact 1} Let  $\bar{x}=\frac{1}{n}\sum_{i=1}^{n}x^i$ and $\hat{x}^i=\sum_{j=1}^{n}c_{ij}x^j$. Then, the following inequality holds
	\begin{align}\label{fact1}
		\Big(\sum_{i=1}^{n}\|\hat{x}^i-\bar{x}\|^2\Big)^{\frac{1}{2}}\leqslant |\lambda|\Big(\sum_{i=1}^{n}\|x^i-\bar{x}\|^2\Big)^{\frac{1}{2}}.
	\end{align}

If Assumptions \ref{ass1} and  \ref{asssss2} hold, Fact 1 implies that each  \textit{Average Consensus} update brings the iteration variables closer to their average $\bar{x}$. For convenience, we define $k_0$ to be the smallest positive integer such that
		$|\lambda|\leqslant[k_0/(k_0+1)]^2$.
	Clearly,  $k_0=\lceil(|\lambda|^{-\frac{1}{2}}-1)^{-1}\rceil$. The next assumption is on the constraint set of $\mathcal{X}$, which is a standard requirement for analyzing FW methods.
	
\begin{assumption}[Constraint set]\label{assumption3}
	$\mathcal X$ is a convex and compact set with diameter $D$, i.e., there is some constant $D>0$ such that $\|x-x'\|\leqslant D$ for all $x,x'\in\mathcal{X}$.
\end{assumption}	
%

\begin{assumption}[$L$-smoothness]\label{ass2}
	Functions $F_i(x)$ and $f_i(x,\xi^i)$  are $L$-smooth with respect to $x$ for all  $\xi^i$ and $i\in\mathcal{N}$.
\end{assumption}
%

\begin{assumption}[Bounded stochastic gradients]\label{assup5}
	The variance of the stochastic gradient $\nabla f_i(x,\xi^i)~(x\in\mathcal{X},~ i\in\mathcal{N})$ is bounded, that is,
		$\mathbb{E}[\|\nabla F_i(x)-\nabla f_i(x,\xi^i)\|^2]\leqslant\delta^2.$
\end{assumption}
Assumption \ref{assup5} is standard for stochastic FW algorithms. The bound will be frequently used for convergence analysis.

\textbf{Fact 2} Suppose Assumptions  \ref{assumption3} and \ref{assup5} hold. There exists a constant $G>0$ such that $\mathbb{E}[\|\nabla f_i(x,\xi^i)\|^2]\leqslant G^2$ and $\mathbb{E}[\|\nabla f_i(x,\xi^i)\|]\leqslant G$.
\begin{proof}
	It follows from the Jensen's inequality that
	\begin{align}\label{fact3eq1}
		\mathbb{E}[\|\nabla F_i(x)\!-\!\nabla f_i(x,\xi^i)\|^2]
	&	\geqslant(\mathbb{E}[\|\nabla f_i(x,\xi^i)-\nabla F_i(x)\|])^2\nonumber\\
		&\geqslant\|\mathbb{E}[\nabla\! f_i(x,\!\xi^i)]\!-\!\nabla\! F_i(x)\|^2\!.
	\end{align}
	 Thus, it follows from  (\ref{fact3eq1}) and Assumption \ref{assup5} that $\|\mathbb{E}[\nabla f_i(x,\xi^i)]-\nabla F_i(x)\|^2\leqslant\delta^2$, i.e., $\|\mathbb{E}[\nabla f_i(x,\xi^i)]-\nabla F_i(x)\|\leqslant\delta$. In addition, we also obtain that $\mathbb{E}[\nabla f_i(x,\xi^i)]$ is bounded because $\nabla F_i(x)$ is bounded followed by Assumption \ref{assumption3}. In the meanwhile,  we have
	$
		\|\nabla F_i(x)\|^2-2\mathbb{E}[\nabla f_i(x,\xi^i)]^{\mathrm{T}}\nabla F_i(x)+\mathbb{E}[\|\nabla f_i(x,\xi^i)\|^2]\leqslant \delta^2
	$ from (\ref{fact3eq1}). Hence, $\mathbb{E}[\|\nabla f_i(x,\xi^i)\|^2]$ has an upper bound, which implies that there is a scalar $G_i$ that $\mathbb{E}[\|\nabla f_i(x,\xi^i)\|^2]\leqslant G^2_i$. Because $G^2_i\geqslant\mathbb{E}[\|\nabla f_i(x,\xi^i)\|^2]\geqslant(\mathbb{E}[\|\nabla f_i(x,\xi^i)\|])^2$, it has $(\mathbb{E}[\|\nabla f_i(x,\xi^i)\|])^2\leqslant G^2_i$, that is $\mathbb{E}[\|\nabla f_i(x,\xi^i)\|]\leqslant G_i$. Let  $G:=\mathop{\mathrm{max}}\limits_{i\in\mathcal{N}}\{G_i\}$. We have that $(\mathbb{E}[\|\nabla f_i(x,\xi^i)\|])^2\leqslant G^2$ and $\mathbb{E}[\|\nabla f_i(x,\xi^i)\|]\leqslant G$.
\end{proof}

\begin{remark}
	 This paper makes weaker assumptions on objective functions. Specifically, compared with \cite{pol2012EHazan}, \cite{Allerton2016SJReddi}  and \cite{ijna2021MWeber}, we do not require  the $Q$-Lipschitz continuity of $f_i(x,\xi^i)$ in \eqref{ss1}, i.e., $\|\nabla f_i(x,\xi^i)\|\leqslant Q$ for all $\xi^i$, where $Q$ is a positive constant.
\end{remark}

\subsection{Convergence rate for convex stochastic optimization}

This subsection is dedicated to the performance analysis of Algorithm \ref{alg:1}. Let us start by defining the following auxiliary vectors
\begin{align}
	\bar{x}_k:=\frac{1}{n}\sum_{i=1}^{n}x_k^i,~~~\bar{y}_k:=\frac{1}{n}\sum_{i=1}^{n}y_k^i,~~~
	\bar{P}_k:=\frac{1}{n}\sum_{i=1}^{n}\nabla F_i(\hat{x}_k^i).\nonumber
\end{align}
We begin our analysis by characterizing the behavior of $\{\hat{x}_k^i\}$ for all $i\in\mathcal{N}$ in the next lemma.
\begin{lemma}\label{lemmama1}
Suppose Assumptions \ref{ass1}-\ref{assumption3} hold. Let $\eta_k=\frac{2}{k+2}$. Then, for any $i\in\mathcal{N}$ and $k\geqslant 1$, we have
	\begin{align}\label{lem1equ}
		\|\hat{x}_k^i-\bar{x}_k\|\leqslant \frac{2C_1}{k+2},
	\end{align}
	where $C_1=k_0\sqrt{n}D$.
\end{lemma}
The proof is presented in Appendix \ref{proofB}. Lemma \ref{lemmama1} shows that $\|\hat{x}_k^i-\bar{x}_k\|=\mathcal{O}(1/k)$, which implies that $\|\hat{x}_k^i-\bar{x}_k\|$  converges to zero  as $k\rightarrow\infty$. By selecting appropriate step sizes $\gamma_k$ and $\eta_k$, we establish the boundedness of $\|p_k^i-\bar{y}_k\|^2$ for all $i\in\mathcal{N}$ in the next lemma.

\begin{lemma}\label{lemma45}
	Suppose Assumptions \ref{ass1}-\ref{assup5} hold. Choose the step sizes $\gamma_k=\frac{2}{k+1}$ and $\eta_k=\frac{2}{k+2}$. Then, for any $i\in\mathcal{N}$ and $k\geqslant 1$, it holds
	\begin{align}\label{lemma45eq1}
		\mathbb{E}[\|p_k^i-\bar{y}_k\|^2]\leqslant& \frac{4C_2}{(k+2)^2},
	\end{align}
 where $C_2=k_0^3(4n)^{k_0}n(12L^2(D+2C_1)^2+12(G^2+\hat{\psi}))$,  $\hat{\psi}=\mathop{\mathrm{max}}\limits_{i\in\mathcal{N}}\{\|y_1^i\|^2,4L(D+2C_1)\psi+4G\psi+8G^2+8L^2(D+2C_1)^2\}$,  $\psi=\mathop{\mathrm{max}}\limits_{i\in\mathcal{N}}\{\|y_1^i\|,2G+2L(D+2C_1)\}$.
\end{lemma}
The proof is presented in Appendix \ref{proofC}.  In order to prove the convergence of Algorithm \ref{alg:1}, we provide the following lemma.

\begin{lemma}\label{lemmalemma34}
Suppose Assumptions \ref{ass1}-\ref{assup5} hold. Then,

(a) the conditional expectation of $\|\bar{P}_k-\bar{y}_k\|^2$ satisfies
\begin{align}
\mathbb{E}_k[\|\bar{P}_k-\bar{y}_k|\mathcal{F}_k\|^2]&\leqslant	(1-\gamma_k)\|\bar{P}_{k-1}-\bar{y}_{k-1}\|^2\nonumber\\
&\quad+6L^2(D+2C_1)^2\eta_{k-1}^2+3\gamma_{k}^2\delta^2\nonumber
\end{align}
for any $k\geqslant 2$,

(b) taking the step sizes $\gamma_k=\frac{2}{k+1}$ and $\eta_k=\frac{2}{k+2}$, the  expectation of  $\|\bar{P}_k-\bar{y}_k\|^2$ satisfies
\begin{align}\label{lemma44eq1}
	\mathbb{E}[\|\bar{P}_k-\bar{y}_k\|^2]\leqslant\frac{C_3}{k+2}
\end{align}
 for any $k\geqslant 1$, where $C_3=24L^2(D+2C_1)^2+12\delta^2$.
\end{lemma}
The proof is presented in Appendix \ref{proofD}. Lemma \ref{lemmalemma34} asserts that  the expectation of $\|\bar{P}_k-\bar{y}_k\|^2$ converges to zero as $k\rightarrow\infty$. Leveraging Lemma \ref{lemma45} and Lemma \ref{lemmalemma34}, the boundedness of $\mathbb{E}[\|\nabla F(\bar{x}_k)-p_k^i\|^2]$ can be derived in the next lemma.

\begin{remark}
	 Consider the  error $\epsilon_k=\bar{P}_k-\bar{y}_k$, which measures the difference incurred by using $y^i_{k}$ as the update direction instead of the correct yet unknown direction $\nabla F_i(\hat{x}_k^i)$ for each agent $i$. Lemma \ref{lemmalemma34}  suggests that $\mathbb{E}[\|\epsilon_k\|^2]$ decreases over iterations, that is, the noise of the stochastic gradient approximation diminishes as the number of iterations increases.
\end{remark}

\begin{remark}
	The conditions in Lemma \ref{lemmalemma34} can be guaranteed by choosing  $\gamma_k=A/(k+t_0)$ and $\eta_k=B/(k+t_0+1)$ according to Lemma \ref{lemmalemma1},  where $A>1$, $B\geqslant 0$ and $t_0$ is a constant. Typical choices are $\gamma_k=2/(k+1)$ and $\eta_k=2/(k+2)$. 
\end{remark}

\begin{lemma}\label{lem1}
Suppose Assumptions \ref{ass1}-\ref{assup5} hold. Take the step sizes $\gamma_k=\frac{2}{k+1}$ and $\eta_k=\frac{2}{k+2}$. Then, for any $i\in\mathcal{N}$ and $k\geqslant 1$, we have
	\begin{align}\label{lemmaequal1}
		\mathbb{E}[\|\nabla F(\bar{x}_k)-p_k^i\|^2]
		\leqslant\frac{12L^2C^2_1+3C_3+12C_2}{k+2}.
	\end{align}
\end{lemma}
The proof is presented in Appendix \ref{proofF1}. Making use of Lemma \ref{lem1}, the convergence rate of Algorithm \ref{alg:1} is established.

\begin{theorem}\label{the1}
	Suppose Assumptions \ref{ass1}-\ref{assup5} hold. The function $F$ is convex. Choose the step sizes $\gamma_k=\frac{2}{k+1}$ and $\eta_k=\frac{2}{k+2}$. Then, for any $k\geqslant 1$, it holds
	\begin{align}\label{the111}
	\mathbb{E}[F(\bar{x}_k)]-F(x^*)
	\leqslant\frac{C_4}{(k+3)^{\frac{1}{2}}},
\end{align}
where $C_4=\mathrm{max}\{\sqrt{3}(F(\bar{x}_1)-F(x^*)),2LD^2+2D\sqrt{12L^2C^2_1+3C_3+12C_2}\}$.
\end{theorem}
The proof is presented in Appendix \ref{proofH}.
\begin{remark}
	Theorem \ref{the1} implies that the  expected suboptimality  $\mathbb{E}[F(\bar{x}_k)]-F(x^*)$ of the iteration variables generated by DMFW converges to zero at least at a sublinear rate of $\mathcal{O}(k^{-\frac{1}{2}})$, coinciding with  that of MSHFW \cite{acc2021ZAkhtar}. By using the Markov inequality, we can also obtain that 
	$
	\mathbb{P}\{F(\bar x_k)-F(x^*)\geqslant C_4k^{-\alpha}\}\leqslant\frac{\mathbb{E}[F(\bar x_k)]-F(x^*)}{C_4k^{-\alpha}}\leqslant k^{\alpha}/(k+3)^{\frac{1}{2}}=\mathcal{O}(k^{\alpha-\frac{1}{2}})
$, where $0<\alpha<\frac{1}{2}$. That is, $\mathbb{P}\{F(\bar x_k)-F(x^*)<C_4k^{-\alpha}\}\geqslant1-k^{\alpha}/(k+3)^{\frac{1}{2}}$. Hence, $\mathbb{P}\{F(\bar x_k)-F(x^*)<C_4k^{-\alpha}\}=1$ when $k\rightarrow\infty$, which indicates that  $F(\bar x_k)$ converges to $F(x^*)$ with probability 1.
\end{remark}

\subsection{Convergence rate for nonconvex optimization}	
This subsection provides the convergence rate of the proposed DMFW for (\ref{ss1}) with nonconvex objective functions.
To show the convergence performance of DMFW for nonconvex case, we introduce the  FW-gap, which is defined as
\begin{align}\label{zx1}
	g_k=\mathop{\mathrm{max}}\limits_{x\in\mathcal{X}}\langle\nabla F(\bar{x}_k),\bar{x}_k-x\rangle.
\end{align}
According to \eqref{zx1}, the variable  $\bar{x}_k$  is a stationary point to \eqref{ss1} if $g_k=0$ \cite{tac2017HWai}. Hence, $g_k$ can be regarded as a measure of the stationarity of variable $\bar{x}_k$. Since the set $\mathcal{X}$ is compact, we assume that the set of stationary points $\mathcal{X}^*$ of \eqref{ss1} is nonempty and the function $F(x)$ is finite  on $\mathcal{X}^*$. The convergence result is presented in the following theorem. 	
\begin{theorem}\label{the2}
	Consider the DMFW algorithm.  Suppose Assumptions \ref{ass1}-\ref{assup5} hold. $F$ is possibly nonconvex. If $\gamma_k=\frac{2}{k+1}$ and $\eta_k=\frac{2}{k+2}$, the FW-gap satisfies
	\begin{align}
	\mathbb{E}\bigg[\mathop{\mathrm{min}}\limits_{k\in[1,K]}g_k\bigg]&\leqslant\frac{1}{\mathrm{log}_2(K)-1}\bigg(\mathbb{E}[F(\bar{x}_1)-F(x^*)]+4LD^2\nonumber\\
		&\quad+2D\beta\sqrt{12L^2C^2_1+3C_3+12C_2}\bigg),\nonumber
\end{align}
where $K=2^m$ ($m\in\mathbb{Z}_+$) and $\beta$ is a constant such that $\sum_{k=1}^{2^m}\frac{2}{(k+2)^{1.5}}\leqslant \beta$.
\end{theorem}
The proof is presented in Appendix \ref{proofI}.
\begin{remark}
	Theorem \ref{the2} indicates that the convergence rate of the proposed DMFW is $\mathcal{O}(1/\mathrm{log}_2(k))$ for nonconvex objective functions. It is worth mentioning that the obtained result is novel compared with previous nonconvex stochastic FW studies, even in centralized setting. For instance, \cite{Allerton2016SJReddi} and \cite{ijna2021MWeber} proposed centralized stochastic FW algorithms for stochastic nonconvex problems by using minibatch method. The proposed DMFW only requires selecting a small fixed batch data (as small as 1) to compute the stochastic gradient to achieve convergence. In addition, the proposed DMFW relaxes the assumption of $Q$-Lipschitz continuity of $f_i(x,\xi^i)$  with respects to $x$, which is required in e.g. \cite{Allerton2016SJReddi} and \cite{ijna2021MWeber}.
\end{remark}

\begin{remark}
	There are two challenges in the convergence analysis of the proposed algorithm. On one hand,  our method needs to deal with the consensus error among different agents compared with the solutions in \cite{panips2019ACutkosky} and \cite{acc2021ZAkhtar}. On the other hand, the introduced local gradient by using recursive momentum makes it more challenging to handle the consensus error compared with the method in \cite{AISTATS2019JXie}.
\end{remark}



\section{Numerical Tests}
\subsection{Binary classification}
In this subsection, several numerical experiments are provided on the binary classification  with an $l_2$-norm ball constraint (i.e., $\mathcal{X}=\{x~|~\|x\|\leqslant \frac{D}{2}\}$). We solve the problem with Algorithm \ref{alg:1} (DMFW) and compare it against SFW \cite{jmlr2020AMokhtari}, MSHFW \cite{acc2021ZAkhtar} and DeFW \cite{tac2017HWai} as baselines. We consider two cases where  objective functions are convex and nonconvex, respectively. We use three different public  datasets, which are summarized in Table \ref{t2}.  In the experiment for SFW, MSHFW and DMFW, we compute a stochastic gradient by using $1\%$ of data per iteration; in the experiment for deterministic algorithm DeFW, we  use full data to compute the gradient. Distributed algorithms DeFW and DMFW are applied over a connected network $\mathcal{G}$ of 5 agents with a doubly stochastic adjacency matrix $C$ to solve the problem. The communication topology is demonstrated in Fig. 1.

\subsubsection{Binary classification with convex objective functions}\label{cas1}
We consider a popular logistic regression for binary classification with convex objective functions as follows:
\begin{align}
	\mathop{\mathrm{min}}\limits_{x\in\mathcal{X}}F(x)&=\frac{1}{n}\sum_{i=1}^nF_i(x),\nonumber\\
	F_i(x)&=\frac{1}{m_i}\sum_{i=1}^{m_i}\mathrm{ln}[1+\mathrm{exp}(-b_i\langle a_i,x\rangle)],\nonumber
\end{align}
where $(a_i,b_i)$ represents the (feature, label) pair of datum $i$, $m_i$ is the total number of training samples of agent $i$, and $n$ is the total number of agents over network $\mathcal{G}$. Here, we set $n=5$ for DeFW and DMFW. Note that SFW and MSHFW are centralized algorithms, so $n=1$ for these two algorithms.  As benchmarks, we choose SFW with step sizes $\gamma_k=2/(k+8)$ and $\rho_k=4/(k+8)^{\frac{2}{3}}$  in \cite{jmlr2020AMokhtari}, MSHFW with step sizes $\gamma_k=2/(k+1)$ and $\eta_k=2/(k+2)$  in \cite{acc2021ZAkhtar}, DeFW with step size $\gamma_k=2/(k+1)$ in \cite{tac2017HWai}. The step sizes of the algorithm DMFW are $\gamma_k=2/(k+1)$ and $\eta_k=2/(k+2)$. The constraint is selected as $\mathcal{X}=\{x~|~\|x\|\leqslant 5\}$.


As described in Lemma \ref{lemmalemma34}, the proposed algorithm can attenuate the noise in gradient approximation, i.e., $\mathbb{E}[\|\bar{P}_k-\bar{y}_k\|^2]=\mathcal{O}(1/k)$. This is verified by our simulation result in Fig. 1. It can be observed that $\|\bar{P}_k-\bar{y}_k\|^2$ generally decreases with the increase of iterations, which is consistent with our theoretical analysis.

We evaluate the algorithm in terms of the FW-gap, which is defined as
\begin{align}
		g_k=&\mathop{\mathrm{max}}\limits_{x\in\mathcal{X}}\langle\nabla F(\bar{x}_k),\bar{x}_k-x\rangle.\nonumber
\end{align}
For different datasets, Fig. 2 shows the FW-gap of  SFW, MSHFW, DeFW and DMFW.  It is shown that DMFW has comparable convergence performance compared with the distributed deterministic algorithm DeFW, although DMFW uses less data. This implies that the local gradient estimate $y^i_k$ in DMFW may be a better candidate for approximating the gradient $\nabla F_i(x)$ comparing to the unbiased gradient estimate $\nabla f_i(x,\xi^i)$. Comparing stochastic methods (MSHFW, SFW and DMFW),  DMFW and MSHFW  outperform SFW, especially on dataset \textit{w8a}, which is consistent with the theoretical result.

\begin{remark}
	 Stochastic algorithms only require a certain amount of data to be obtained randomly at each time to achieve convergence, unlike deterministic algorithms, which need to obtain the whole data at one time prior to the start of the algorithms.   Hence, finite-sum problem can also be solved by stochastic algorithms. In Test A, the random variable $\xi^i$  denotes  training examples obtained randomly by agent $i$ at each time, we use $1\%$ of data to estimate a stochastic gradient $\nabla f_i(x,\xi^i)$ for stochastic algorithms at each iteration, not the full gradient in  finite-sum setting. 
\end{remark}

%
 

\begin{table}
 \begin{center}
\caption{REAL DATA FOR BLACK-BOX BINARY CLASSIFICATION}\label{t2}
\begin{tabular}{c|c|c|c}
\hline
datasets& \#samples & \#features & \#classes\\\hline
\textit{covtype.binary} & 581012 & 54 & 2\\\hline
\textit{a9a} & 32561 & 123 & 2\\\hline
\textit{w8a} & 64700 & 300 &2\\
\hline
\end{tabular}
\end{center}
\end{table}	

\begin{figure}[!t]
\centering
\includegraphics[width=2.5in]{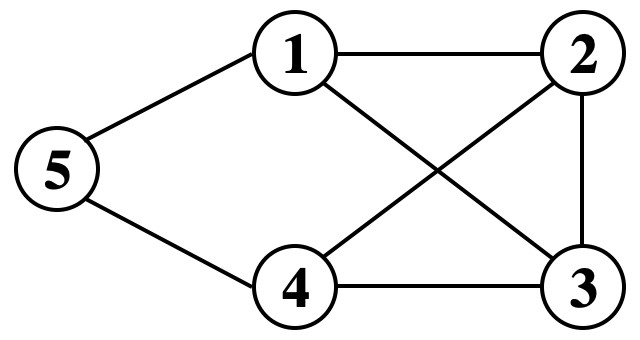}%
\caption{Multi-agent communication topology.}
\label{fig_sim}
\end{figure}

 \begin{figure}[!t]
\centering
\includegraphics[width=2.5in]{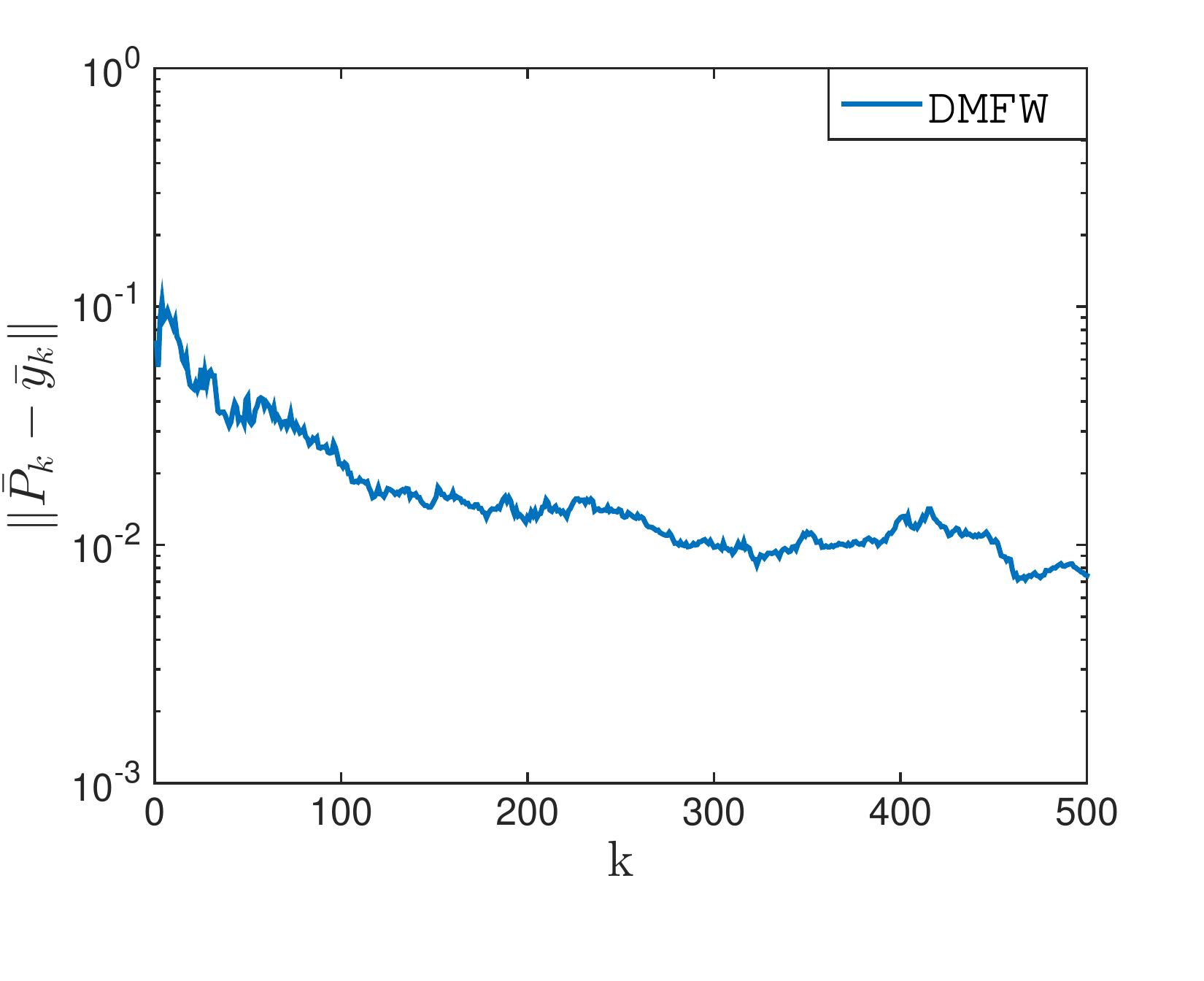}%
\caption{The error $\|\bar{P}_k-\bar{y}_k\|$ of DMFW on $a9a$ dataset.}
\label{fig_sim}
\end{figure}


\begin{figure*}[!t]
\centering
\subfloat[]{\includegraphics[width=2in]{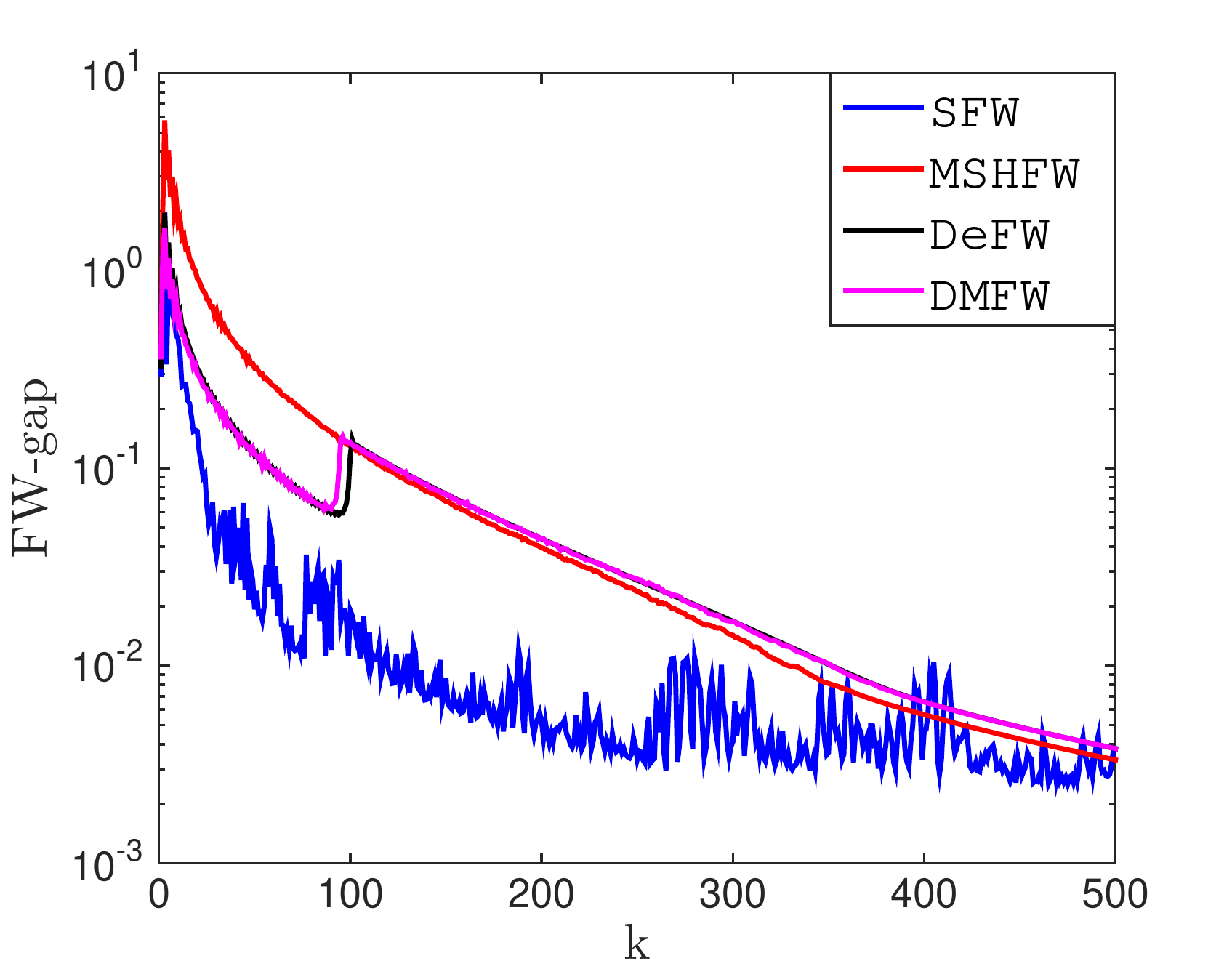}%
\label{fig2_forth_case}}
\hfil
\subfloat[]{\includegraphics[width=2in]{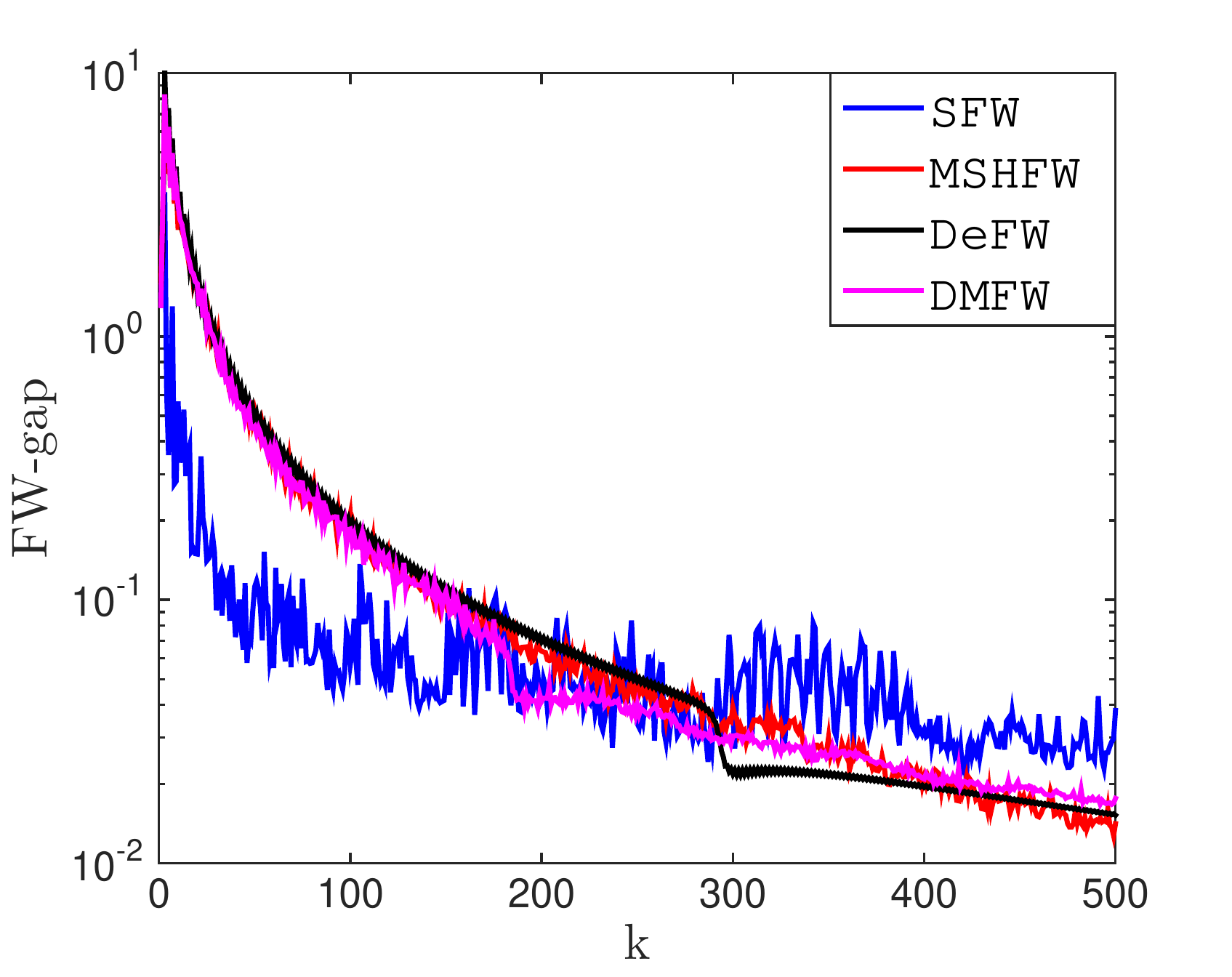}%
\label{fig2_fifth_case}}
\hfil
\subfloat[]{\includegraphics[width=2in]{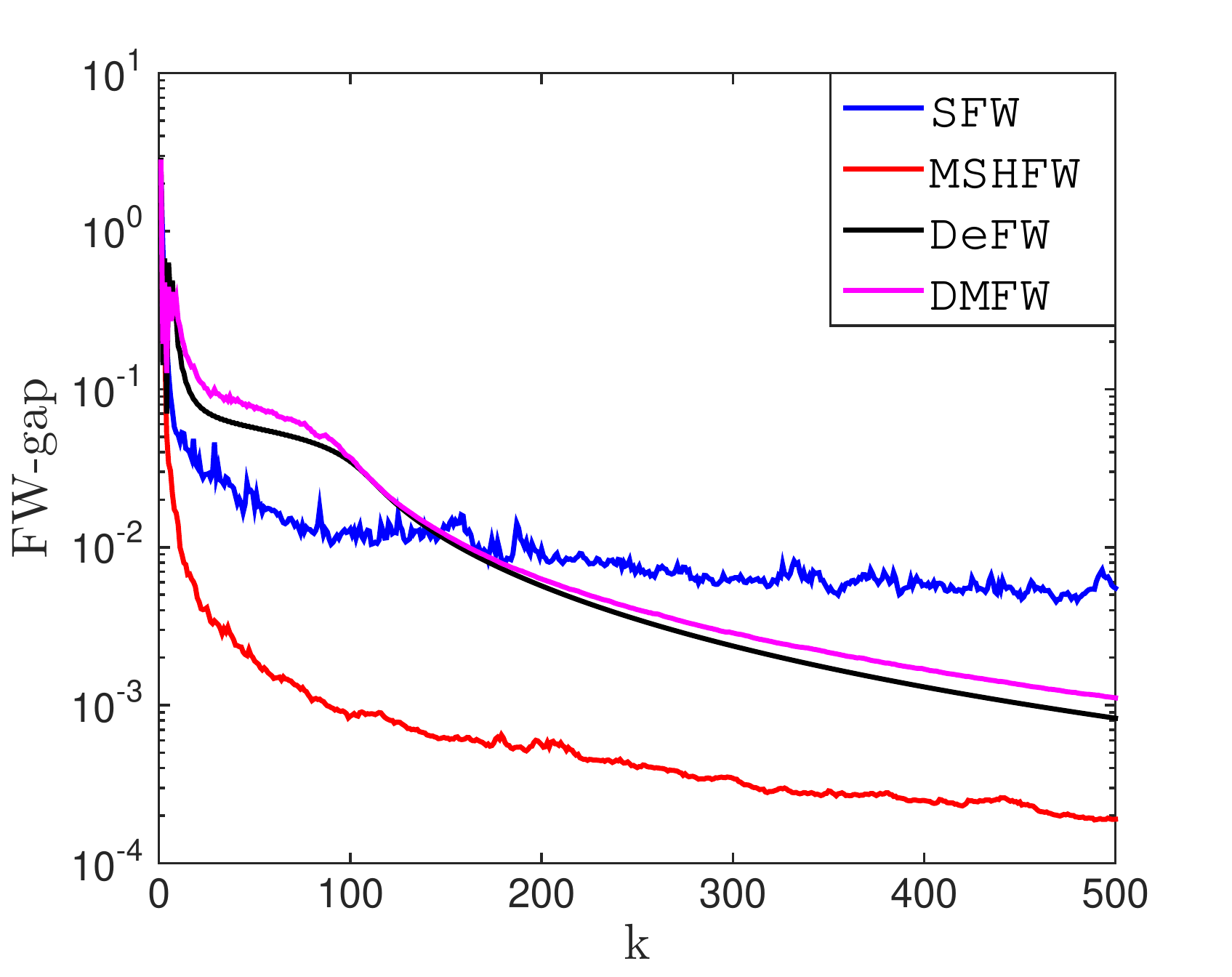}%
\label{fig2_sixth_case}}
\caption{The comparison between SFW, MSHFW, DeFW and DMFW on three datasets. (a) \textit{covtype.binary} dataset. (b) \textit{a9a} dataset. (c) \textit{w8a} dataset.}
\label{fig_sim}
\end{figure*}

\subsubsection{Binary Classification with Nonconvex Objective Functions} We
consider a binary classification with nonconvex objective function as follows:
\begin{align}\label{ss3}
	\mathop{\mathrm{min}}\limits_{x\in\mathcal{X}}~F(x)&=\frac{1}{n}\sum_{i=1}^nF_i(x),\nonumber\\
	F_i(x)&=\frac{1}{m_i}\sum_{i=1}^{m_i}\frac{1}{1+\mathrm{exp}(b_i\langle a_i,x\rangle)}+\lambda_1\|x\|^2
\end{align}
where $(a_i,b_i)$, $m_i$, $n$ and the constraint $\mathcal{X}$ have the same definitions as those in \ref{cas1}, and $\lambda_1=5\times 10^{-6}$. It is obvious that the objective function $F$ is a nonconvex function. As benchmarks, we choose DeFW with step size $\gamma_k=\frac{1}{k^{0.5}}$ as mentioned in \cite{tac2017HWai}.  Algorithms SFW and MSHFW use the same step sizes in Section \ref{cas1}. The step sizes of the algorithm DMFW are $\gamma_k=\frac{2}{k+1}$ and $\eta_k=\frac{2}{k+2}$. Note that SFW and MSHFW are not proved to converge for the nonconvex problem (\ref{ss3}). We implement these two algorithms only for comparison purpose.

Fig. 3 shows the FW-gap of SFW, MSHFW, DeFW and DMFW for solving nonconvex problem \eqref{ss3}. From the results, it can be observed that stochastic algorithms (SFW, DMFW and MSHFW)  perform better compared to deterministic algorithm DeFW in all tested datasets.  This indicates that the stochastic FW algorithms are more efficient than the deterministic FW algorithms in solving nonconvex problems \eqref{ss3}. Comparing DMFW with centralized algorithms MSHFW and SFW, DMFW slightly outperforms SFW, especially in datasets $a9a$, but is slower than MSHFW.
 

\begin{figure*}[!t]
\centering
\subfloat[]{\includegraphics[width=2in]{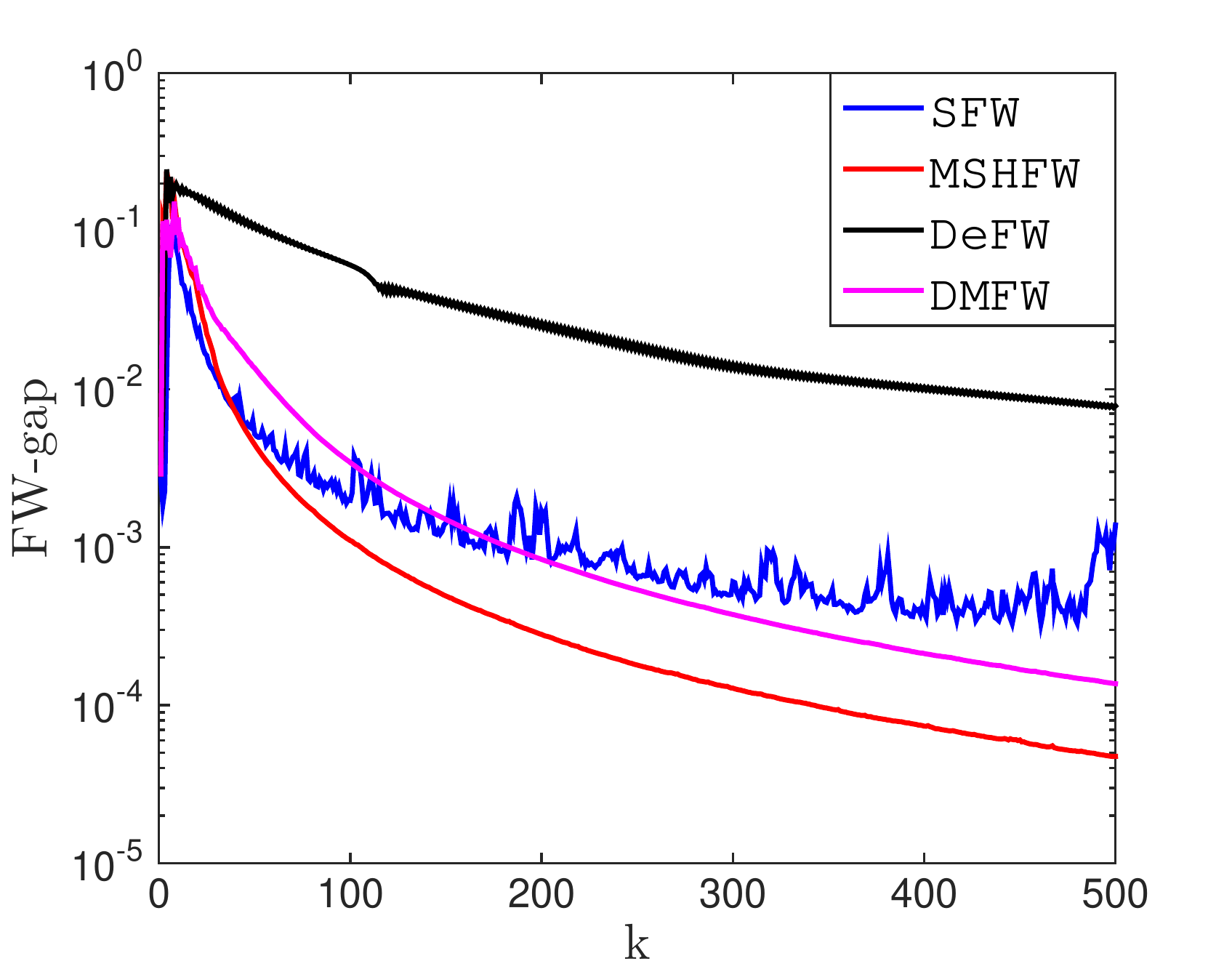}%
\label{fig4_forth_case}}
\hfil
\subfloat[]{\includegraphics[width=2in]{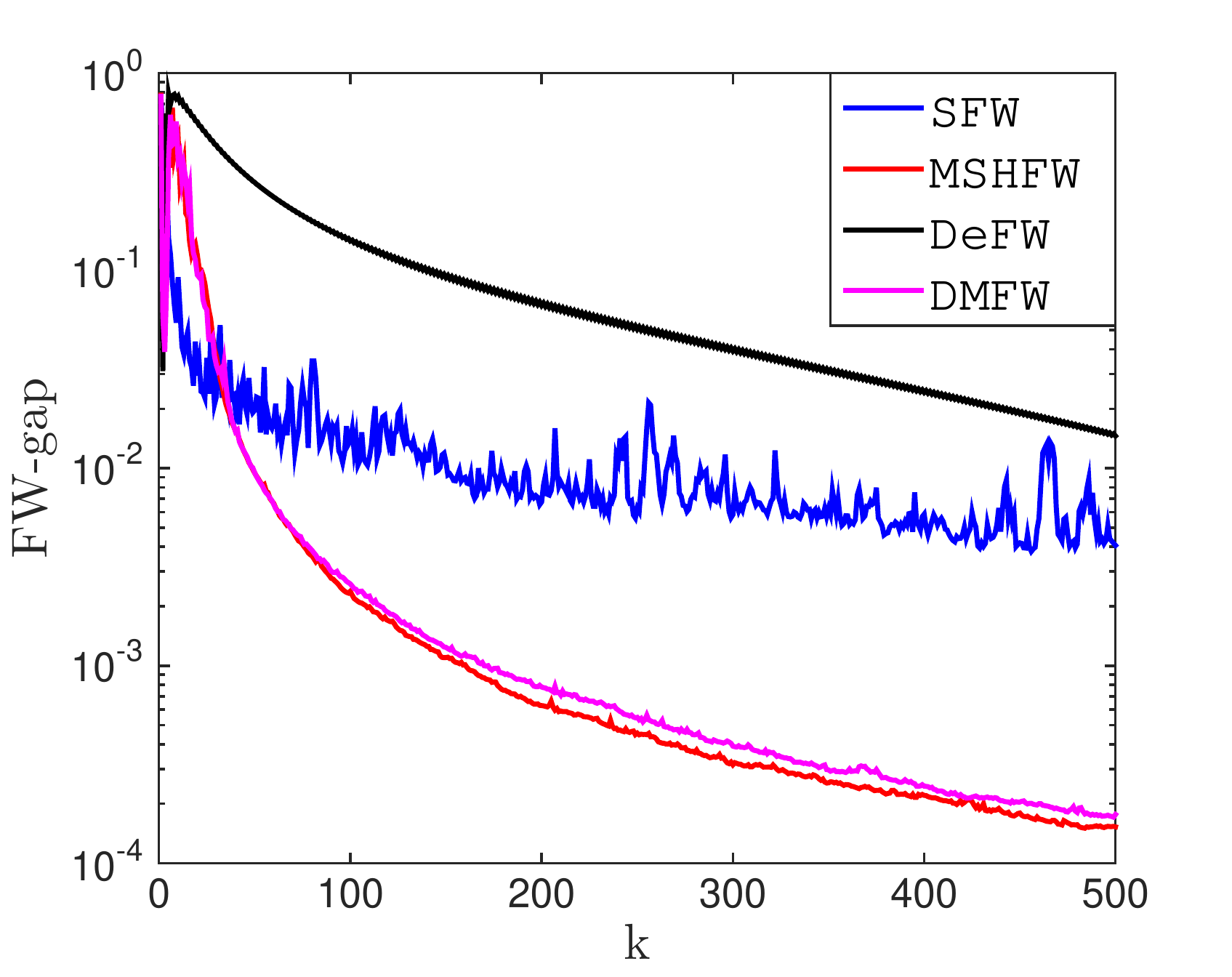}%
\label{fig4_fifth_case}}
\hfil
\subfloat[]{\includegraphics[width=2in]{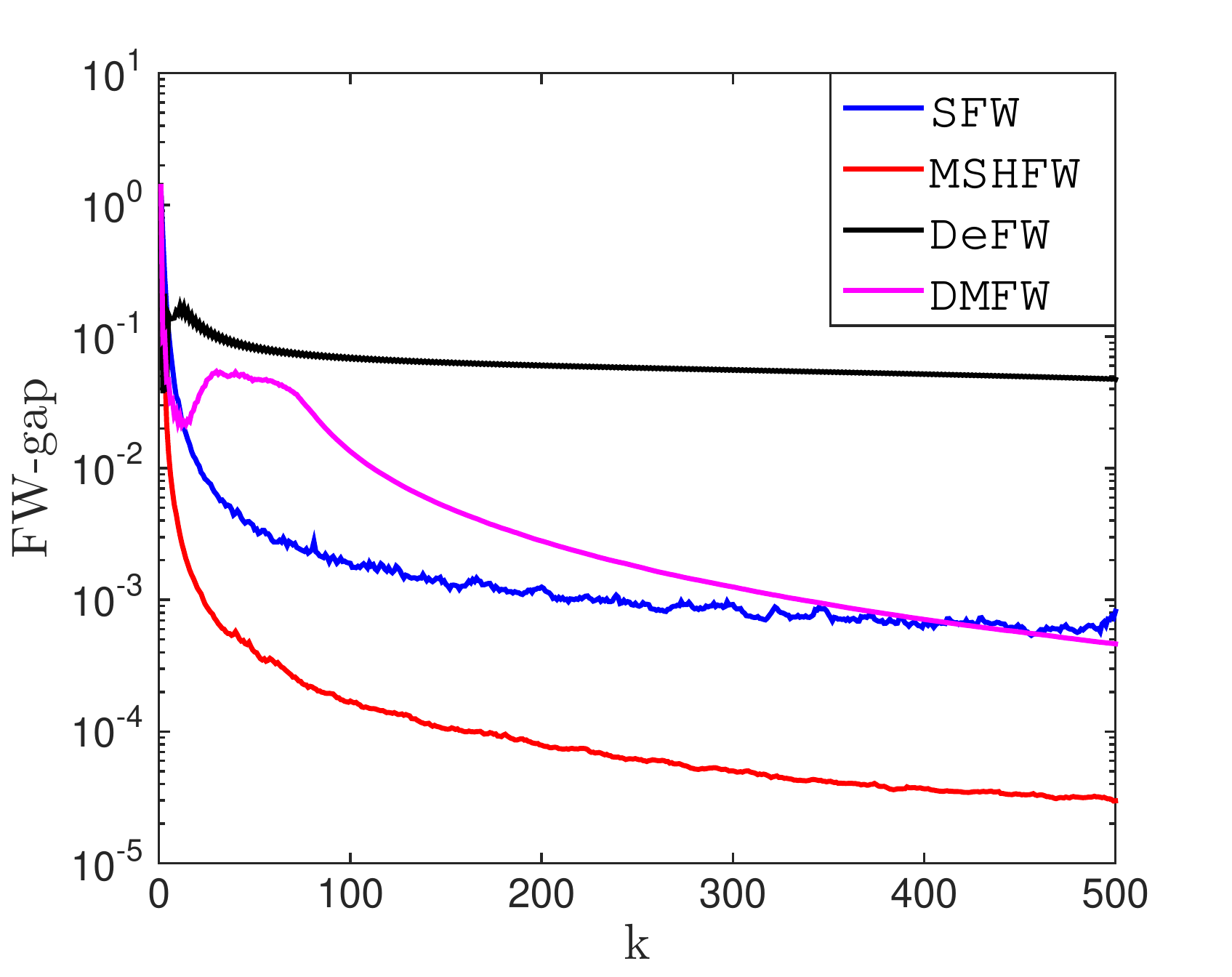}%
\label{fig4_sixth_case}}
\caption{The comparison between SFW, MSHFW, DeFW and DMFW on three datasets. (a) \textit{covtype.binary} dataset. (b) \textit{a9a} dataset. (c) \textit{w8a} dataset.}
\label{fig_sim}
\end{figure*}

\begin{figure*}[!t]
\centering
\subfloat[]{\includegraphics[width=2in]{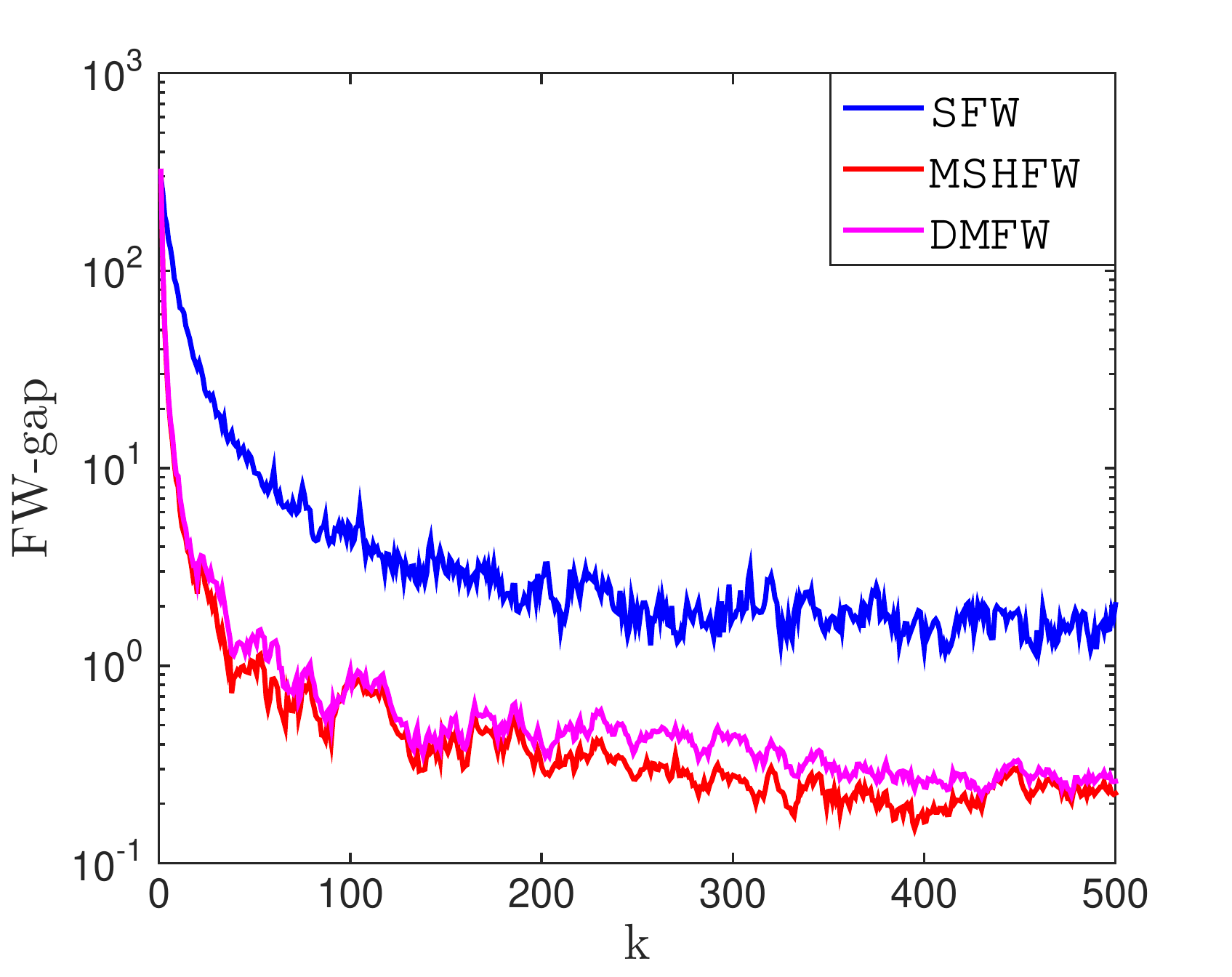}%
\label{fig4_forth_case}}
\hfil
\subfloat[]{\includegraphics[width=2in]{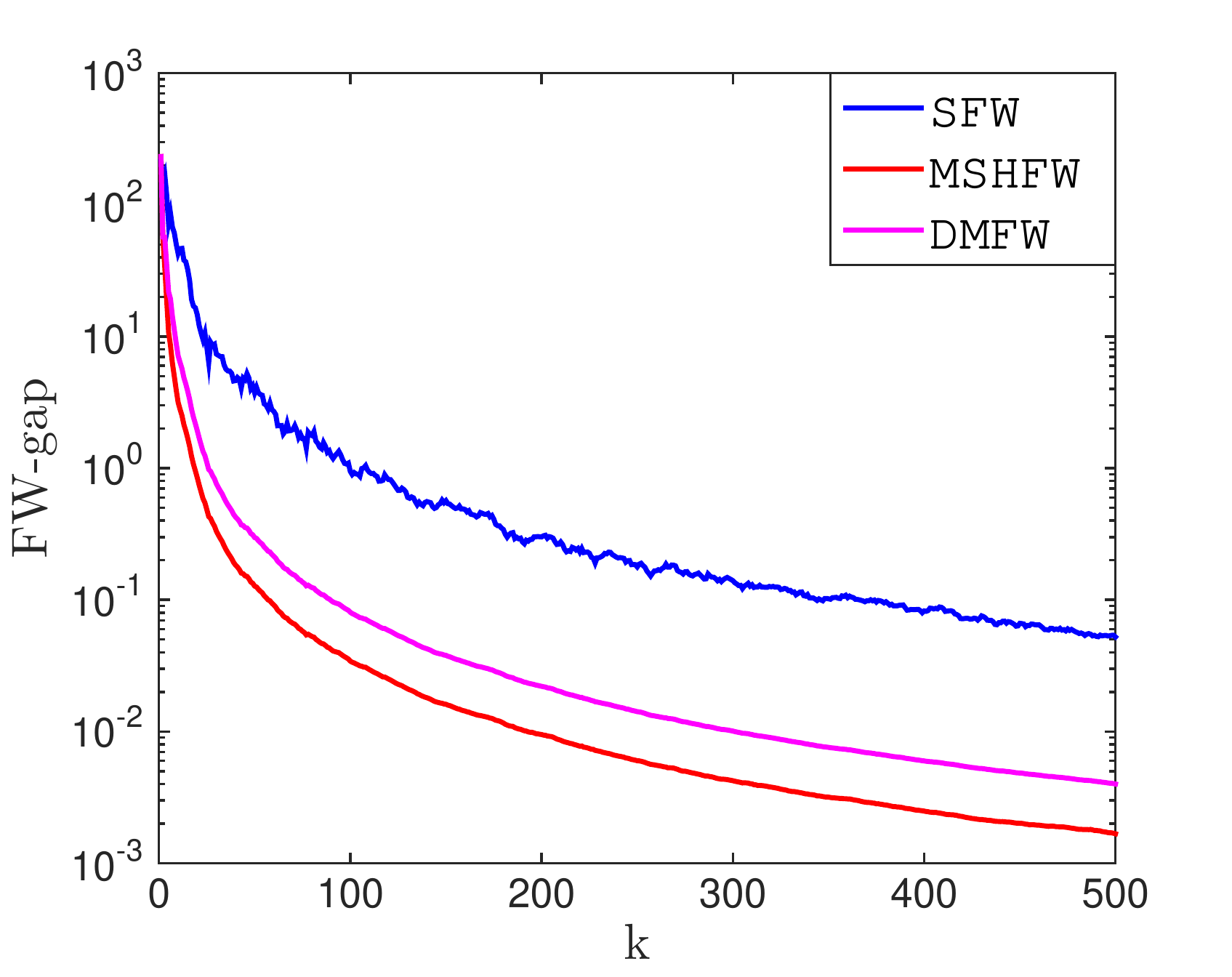}%
\label{fig4_fifth_case}}
\hfil
\subfloat[]{\includegraphics[width=2in]{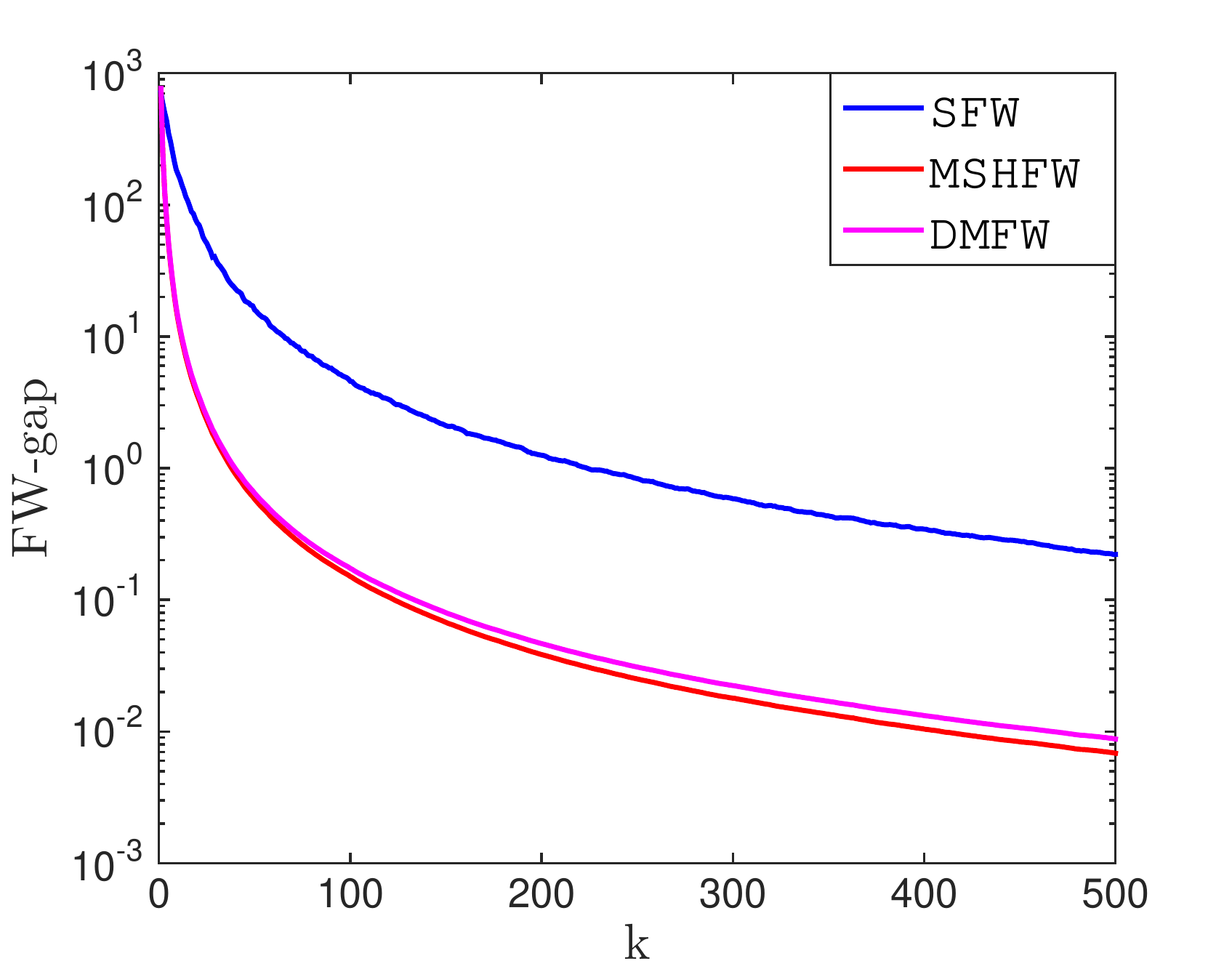}%
\label{fig4_sixth_case}}
\caption{The comparison between SFW, MSHFW and DMFW. (a) $l_1$ norm ball constraint. (b) $l_2$ norm ball constraint. (c) $l_{\frac{5}{4}}$ norm ball constraint.}
\label{fig_sim}
\end{figure*}

\subsection{Stochastic ridge regression}
In this subsection, several numerical experiments are conducted for the stochastic ridge regression. The constraint sets considered include $l_1$-, $l_2$- and $l_{\frac{5}{4}}$- norm balls. We solve the problem with Algorithm \ref{alg:1} (DMFW) and compare it against SFW \cite{jmlr2020AMokhtari} and MSHFW \cite{acc2021ZAkhtar} as baselines.  DMFW is applied over a connected network $\mathcal{G}$ of $n$ agents with a doubly stochastic adjacency matrix $C$ to solve the problem. The stochastic ridge regression is as follows:
\begin{align}
	\mathop{\mathrm{min}}\limits_{x\in\mathcal{X}}F(x)&=\frac{1}{n}\sum_{i=1}^nF_i(x),\nonumber\\
	F_i(x)&=\mathbb{E}_{a_i,b_i}[(a_i^Tx-b_i)^2+\lambda_1\|x\|^2],\nonumber
\end{align}
where $\lambda_1=5\times10^{-6}$ is a penalty parameter,  $n$ is the total number of agents over network $\mathcal{G}$, $(a_i,b_i)$ represents the (feature, label) pair of datum $i$. We assume that each $a_i\in[0.3,0.4]^p$ is uniformly distributed, and $b_i$ is chosen according to $b_i=a_i^Tz_i+\sigma_i$, where $z_i$ is a predefined parameter evenly distributed in $[0,10]^p$, and $\sigma_i$ is independent Gaussian noise with mean value of 0 and variance of 1. Given a pair $(a_i,b_i)$, agent $i$ can compute an estimated gradient of $F_i(x)$: $\nabla f_i(x,a_i,b_i)=2(a^T_ix-b_i)a_i+2\lambda_1x$. Choose $p=50$ and $n=50$. In the experiments, three constraint sets $\{x~|~\|x\|_1\leqslant 5\}$, $\{x~|~\|x\|\leqslant 5\}$ and $\{x~|~\|x\|_{\frac{5}{4}}\leqslant 5\}$ are considered. As benchmarks, we choose SFW with step sizes $\gamma_k=2/(k+8)$ and $\rho_k=4/(k+8)^{\frac{2}{3}}$ as mentioned in \cite{jmlr2020AMokhtari}, MSHFW with step sizes $\gamma_k=2/(k+1)$ and $\eta_k=2/(k+2)$ as mentioned in \cite{acc2021ZAkhtar}. The step sizes of the algorithm DMFW are $\gamma_k=2/(k+1)$ and $\eta_k=2/(k+2)$.

Fig. 4 shows the convergence performances of stochastic algorithms SFW, MSHFW and DMFW with the same parameters. In the simulation, SFW converges slower than DMFW and MSHFW in all cases, which is consistent with the theoretical results. DMFW performs comparable convergence performance to centralized algorithm MSHFW when the constraint set are $l_2$- and $l_{\frac{5}{4}}$- norm balls.
%


\section{Conclusions}
This paper proposed a distributed stochastic Frank-Wolfe algorithm by combining Nesterov's momentum with gradient tracking technique for stochastic optimization problems. As far as we know, this is the first distributed Frank-Wolfe algorithm for solving stochastic optimization problems. For convex objective functions, the proposed method achieves the convergence rate of $\mathcal{O}(k^{-\frac{1}{2}})$, coinciding with that of the centralized stochastic algorithms. In addition, the proposed algorithm achieves the convergence rate of $\mathcal{O}(1/\mathrm{log}_2(k))$ for nonconvex problems under weaker assumptions than existing ones. The efficacy of the proposed algorithm was tested on binary classification problems with convex and nonconvex objective functions. In our future research, it is interesting to consider stochastic optimization with nonsmooth objective functions by using the conditional gradient method.

%
\section{APPENDIX}

\subsection{Proof of Lemma \ref{lemmama1}}\label{proofB}
Before proving Lemma \ref{lemmama1}, we first give the following technical lemmas.

The following lemma is presented in Lemma 2 of the supplementary material for \cite{acc2021ZAkhtar}.
\begin{lemma}\label{lemmalemma1} Let $\phi_k$ be a sequence of real numbers satisfying
	\begin{align}
		\phi_k=\Bigg(1-\frac{A}{(k+t_0)^{r_1}}\Bigg)\phi_{k-1}+\frac{B}{(k+t_0)^{r_2}},\nonumber
	\end{align}
	for some $r_1\in[0,1]$ such that $r_1\leqslant r_2\leqslant 2r_1$, $A> 1$ and $B\geqslant 0$. Then, $\phi_k$ converges to zero at the following rate:
	\begin{align}
		\phi_k\leqslant\frac{H}{(k+t_0+1)^{r_2-r_1}},\nonumber
	\end{align}
	where $H=\mathrm{max}\{\phi_{0}(t_0+1)^{r_2-r_1},\frac{B}{A-1}\}$.
\end{lemma}

\begin{lemma}\label{lemma0} Consider Algorithm \ref{alg:1}. Suppose Assumption \ref{ass1} holds. For all $k=1,2,\cdots,K$, the following relationships are established.

(a) $\frac{1}{n}\sum\limits_{i=1}^{n}s_{k+1}^i=\bar{y}_{k+1};$

(b) $\bar{x}_{k+1}=(1-\eta_k)\bar{x}_k+\eta_k\bar{\theta}_k,$
where $\bar{\theta}_k=\frac{1}{n}\sum_{i=1}^{n}\theta_k^i$.
\end{lemma}

\begin{proof}
(a) From \eqref{algeua1} of Algorithm \ref{alg:1}, we have
	\begin{align}
		\frac{1}{n}\sum\limits_{i=1}^{n}s_{k+1}^i
		&=\frac{1}{n}\sum\limits_{i=1}^{n}\bigg(\sum\limits_{j=1}^{n}c_{ij}s_k^j+y_{k+1}^i-y_{k}^i\bigg)\nonumber\\
		&=\frac{1}{n}\sum\limits_{i=1}^{n}s_k^i+\bar{y}_{k+1}-\bar{y}_k\nonumber\\
		&=\frac{1}{n}\sum\limits_{i=1}^{n}y_1^i-\bar{y}_1+\bar{y}_{k+1}
		=\bar{y}_{k+1},\nonumber
	\end{align}
	where the second equality is due to the fact that matrix $C$ is doubly stochastic. Hence, $\frac{1}{n}\sum\limits_{i=1}^{n}s_{k+1}^i=\bar{y}_{k+1}$.
	
	(b) According to the definitions of $\bar{x}_k$ and $x_k^i$, we have
	\begin{align}
		\bar{x}_{k+1}&=\frac{1}{n}\sum_{i=1}^{n}[(1-\eta_k)\sum\limits_{j=1}^{n}c_{ij}x_k^j+\eta_k\theta_k^i]\nonumber\\
		&=\frac{(1-\eta_k)}{n}\sum_{i=1}^{n}x_k^i+\frac{\eta_k}{n}\sum_{i=1}^{n}\theta_k^i\nonumber\\
		&=(1-\eta_k)\bar{x}_k+\eta_k\bar{\theta}_k,\nonumber
	\end{align}
	where the first equality is due to the fact that matrix $C$ is doubly stochastic. Hence, $\bar{x}_{k+1}=(1-\eta_k)\bar{x}_{k}+\eta_k\bar{\theta}_{k}$.
\end{proof}

Then we  give the proof of Lemma \ref{lemmama1}.
\begin{proof}
	We use mathematical induction to prove this lemma. It follows from  the properties of Euclidean norm that
	$
		\|\hat{x}_k^i-\bar{x}_k\|\leqslant\mathop{\mathrm{max}}\limits_{i\in\mathcal{N}}\|\hat{x}_k^i-\bar{x}_k\|\leqslant(\sum_{i=1}^{n}\|\hat{x}_k^i-\bar{x}_k\|^2)^\frac{1}{2}.$	We next prove the following inequality  by using induction on $k$,
	\begin{align}\label{ex2}
		\bigg(\sum_{i=1}^{n}\|\hat{x}_k^i-\bar{x}_k\|^2\bigg)^\frac{1}{2}\leqslant \frac{2C_1}{k+2}=\frac{2k_0\sqrt{n}D}{k+2}=C_1\eta_k.
	\end{align}
	It is obvious that inequality \eqref{ex2} holds for $k=1$ to $k=k_0-2$.
	
	 For induction step, we assume that \eqref{ex2} holds for some $k\geqslant k_0-2$. According to Lemma \ref{lemma0} (b) and \eqref{algstep5}, definitions of $\hat{x}_k^i$ and $\bar{x}_k$, we have
	\begin{align}\label{ex3}
		&\quad\sum_{i=1}^{n}\big{\|}\hat{x}_{k+1}^i-\bar{x}_{k+1}\big{\|}^2\nonumber\\
		&=\sum_{i=1}^{n}\bigg{\|}\!\sum_{j=1}^{n}c_{ij}(1\!-\!\eta_k)\hat{x}_k^j\!+\!\sum_{j=1}^{n}\!c_{ij}\eta_k\theta_k^j\!-\!(1\!-\!\eta_k)\bar{x}_k\!-\!\eta_k\bar{\theta}_k\bigg{\|}^2\nonumber\\
		&=\sum_{i=1}^{n}\bigg{\|}\sum_{j=1}^{n}c_{ij}[(1-\eta_k)\sum_{h=1}^{n}c_{jh}x_k^h+\eta_k\theta_k^j]\nonumber\\
		&\quad-\frac{1}{n}\sum_{j=1}^{n}[(1-\eta_k)\sum_{h=1}^{n}c_{jh}x_k^h+\eta_k\theta_k^j]\bigg{\|}^2\nonumber\\
		&\leqslant|\lambda|^2\sum_{i=1}^{n}\!\big{\|}(1-\eta_k)(\hat{x}_k^i\!-\!\bar{x}_k)\!+\!\eta_k(\theta_k^i-\bar{\theta}_k)\big{\|}^2,
	\end{align}
 where $\lambda$ is the second largest eigenvalue
of $C$; the last inequality follows from \eqref{fact1}. Now we only need prove  that $\sum_{i=1}^{n}\|(1-\eta_k)(\hat{x}_k^i-\bar{x}_k)+\eta_k(\theta_k^i-\bar{\theta}_k)\|^2$ has an upper bound. We have
	\begin{align}\label{ex4}
		&\quad\sum_{i=1}^{n}\|(1-\eta_k)(\hat{x}_k^i-\bar{x}_k)+\eta_k(\theta_k^i-\bar{\theta}_k)\|^2\nonumber\\
		&\overset{(a)}{\leqslant}\sum_{i=1}^{n}[(1\!-\!\eta_k)^2\|\hat{x}_k^i\!-\!\bar{x}_k\|^2+\!\eta^2_kD^2\!+\!2\eta_k(1\!-\eta_k)D\|\hat{x}_k^i\!-\!\bar{x}_k\| ]\nonumber\\
		&\overset{(b)}{\leqslant}\sum_{i=1}^{n}(\|\hat{x}_k^i-\bar{x}_k\|^2+2\eta_kD\|\hat{x}_k^i-\bar{x}_k\|+\eta^2_kD^2)\nonumber\\
		&\overset{(c)}{\leqslant}C_1^2\eta^2_k+n\eta^2_kD^2+2\eta_k D\sqrt{n}\sqrt{\sum_{i=1}^{n}\|\hat{x}_k^i-\bar{x}_k\|^2}\nonumber\\
		&\leqslant C_1^2\eta^2_k+n\eta^2_kD^2+2D\eta_k^2\sqrt{n}C_1\nonumber\\
		&=\eta_k^2(C_1+C_1{k_0}^{-1})^2\nonumber\\
		&=\bigg(\frac{k_0+1}{k_0}C_1\eta_k\bigg)^2,
	\end{align}
	where $(a)$ follows from Assumption \ref{assumption3}; $(b)$ follows from $1-\eta_k\leqslant 1$; $(c)$ is due to $\sum_{i=1}^{n}\|\hat{x}_k^i-\bar{x}_k\|\leqslant\sqrt{n}\sqrt{\sum_{i=1}^{n}\|\hat{x}_k^i-\bar{x}_k\|^2}$ and the induction hypothesis \eqref{ex2}. Substituting \eqref{ex4}, $\lambda\leqslant\big(\frac{k_0}{k_0+1}\big)^2$ and $\eta_k=\frac{2}{k+2}$ into \eqref{ex3}, we arrive at
	\begin{align}\label{ex5}
		&\quad\sum_{i=1}^{n}\|\hat{x}_{k+1}^i-\bar{x}_{k+1}\|^2\nonumber\\
		&\leqslant\bigg(\frac{2(k_0+1)}{k_0(k+2)}\big(\frac{k_0}{k_0+1}\big)^2C_1\bigg)^2\nonumber\\
		&\leqslant\bigg(\frac{2(k+2)}{(k+2+1)(k+2)}C_1\bigg)^2
		=C^2_1\eta^2_{k+1}
	\end{align}
	where the second inequality is because of  the  monotonically increasing property of function $g(v)=v/(1+v)$ with respect to $v$ over $[0,\infty)$. It follows from \eqref{ex5}  that  $\sum_{i=1}^{n}\|\hat{x}_{k+1}^i-\bar{x}_{k+1}\|_2\leqslant C_1\eta_{k+1}$, that is, \eqref{ex2} holds for iteration $k+1$. Hence, $\|\hat{x}_{k}^i-\bar{x}_{k}\|\leqslant 2C_1/(k+2)$ for all $k\geqslant 1$.
\end{proof}

\subsection{Technical Lemmas}

\begin{lemma}\label{lemma32}
Suppose Assumptions \ref{ass1}-\ref{assumption3} hold. Let $\eta_k=\frac{2}{k+2}$. Then, for any $i\in\mathcal{N}$ and $k\geqslant 1$, we have
\begin{align}\label{lemma32equ}
	\|\hat{x}_{k+1}^i-\hat{x}_k^i\|\leqslant\frac{2(D+2C_1)}{k+2},
\end{align}
where $C_1=k_0\sqrt{n}D$.
\end{lemma}
\begin{proof}
It follows from the definition of $\hat{x}_k^i$ that
	\begin{align}
	&\quad\|\hat{x}_{k+1}^i-\hat{x}_k^i\|\nonumber\\
	&\leqslant\sum_{j=1}^{n}c_{ij}(\|x_{k+1}^j-\hat{x}_k^j\|+\|\hat{x}_k^j-\hat{x}_k^i\|)\nonumber\\
	&\overset{(a)}{=}\sum_{j=1}^{n}c_{ij}(\|\eta_k\theta_k^j-\eta_k\hat{x}_k^j\|+\|\hat{x}_k^j-\bar{x}_k+\bar{x}_k-\hat{x}_k^i\|)\nonumber\\
	&\overset{(b)}{\leqslant}\sum_{j=1}^{n}c_{ij}(\|\hat{x}_k^j-\bar{x}_k\|+\|\hat{x}_k^i-\bar{x}_k\|)+\sum_{j=1}^{n}c_{ij}\|\eta_k(\theta_k^j-\hat{x}_k^j)\|\nonumber\\
	&\overset{(c)}{\leqslant}\sum_{j=1}^{n}c_{ij}(\eta_kD+2C_1\eta_k)
	=\frac{2(D+2C_1)}{k+2},
\end{align}
where $(a)$ follows from \eqref{algstep5}; $(b)$ holds for the triangle inequality;  $(c)$ is because of Assumption \ref{assumption3} and Lemma \ref{lemmama1}.
\end{proof}

\begin{lemma}\label{lemma4444}
Suppose Assumptions \ref{ass1}-\ref{assup5} hold. Choose the step sizes $\gamma_k=\frac{2}{k+1}$ and $\eta_k=\frac{2}{k+2}$. Then, for any $i\in\mathcal{N}$ and $k\geqslant 1$, the variable $y^i_k$ of Algorithm \ref{alg:1} satisfies the following bound:
	\begin{align}\label{lemma44eqw1}
	&\mathbb{E}[\|y_k^i\|]\leqslant\psi,
\end{align}
where  $\psi=\mathop{\mathrm{max}}\limits_{i\in\mathcal{N}}\{\|y_1^i\|,2G+2L(D+2C_1)\}$.
\end{lemma}
\begin{proof}
It is obvious that \eqref{lemma44eqw1} holds when $k=1$. Next let's  discuss the case when $k\geqslant 2$. It follows from \eqref{algstep3} of Algorithm \ref{alg:1} that
\begin{align}
          &\quad\mathbb{E}[\|y_k^i\|]\nonumber\\
          &=\mathbb{E}[\|(1-\gamma_k)y_{k-1}^i+\nabla f_i(\hat{x}_k^i,\xi_k^i)-(1-\gamma_k)\nabla f_i(\hat{x}_{k-1}^i,\xi_k^i)\|]\nonumber\\
          &\leqslant(1\!-\!\gamma_k)\mathbb{E}[\|y_{k-1}^i\|]\!+\!\mathbb{E}[\|\nabla\! f_i(\hat{x}_k^i,\!\xi_k^i)\!-\!(1\!-\!\gamma_k)\nabla\! f_i(\hat{x}_{k-1}^i\!,\xi_k^i)\|]\nonumber\\
          &\leqslant(1-\gamma_k)\mathbb{E}[\|y_{k-1}^i\|]+\gamma_k\mathbb{E}[\|\nabla f_i(\hat{x}_{k-1}^i,\xi_{k}^i)\|]\nonumber\\
          &\quad+\mathbb{E}[\|\nabla f_i(\hat{x}_{k}^i,\xi_{k}^i)-\nabla f_i(\hat{x}_{k-1}^i,\xi_{k}^i)\|]\nonumber\\
         &\overset{(a)}{\leqslant}(1-\gamma_k)\mathbb{E}[\|y_{k-1}^i\|]+G\gamma_k+L\mathbb{E}[\|\hat{x}_{k}^i-\hat{x}_{k-1}^i\|]\nonumber\\
          &\overset{(b)}{\leqslant}(1-\gamma_k)\mathbb{E}[\|y_{k-1}^i\|]+G\gamma_k+L(D+2C_1)\eta_{k-1}\nonumber\\
          &=(1-\frac{2}{k+1})\mathbb{E}[\|y_{k-1}^i\|]+\frac{2G+2L(D+2C_1)}{k+1}
\end{align}
where $(a)$ is due to the $L$-smooth property of function $f_i$ and Fact 2; $(b)$ follows from \eqref{lemma32equ}. By using Lemma \ref{lemmalemma1} ($r_1=1,~r_2=1,~A=2$ and $B=2G+2L(D+2C_1)$), we obtain $
	\mathbb{E}[\|y_k^i\|]\leqslant\psi,$ where $\psi=\mathop{\mathrm{max}}\limits_{i\in\mathcal{N}}\{\|y_1^i\|,2G+2L(D+2C_1)\}$.
\end{proof}

\begin{lemma}\label{lemma4555}
Suppose Assumptions \ref{ass1}-\ref{assup5} hold. Choose the step sizes $\gamma_k=\frac{2}{k+1}$ and $\eta_k=\frac{2}{k+2}$. Then, for any $i\in\mathcal{N}$ and $k\geqslant 1$, it holds
	\begin{align}\label{lemma45eqw1}
	&\mathbb{E}[\|y_k^i\|^2]\leqslant\hat{\psi},
\end{align}
where $\hat{\psi}=\mathop{\mathrm{max}}\limits_{i\in\mathcal{N}}\{\|y_1^i\|^2,4L(D+2C_1)\psi+4G\psi+8G^2+8L^2(D+2C_1)^2\}$.
\end{lemma}
\begin{proof}
It is obvious that \eqref{lemma45eqw1} holds when $k=1$. Next let's  discuss the case when $k\geqslant 2$. From \eqref{algstep3} of Algorithm \ref{alg:1}, we have
\begin{align}\label{qq1}
         &\quad\|y_k^i\|^2
          \leqslant(1-\gamma_k)^2\|y_{k-1}^i\|^2+2(1-\gamma_k)\|\nabla f_i(\hat{x}_k^i,\xi_k^i)\nonumber\\
          &\quad-(1-\gamma_k)\nabla f_i(\hat{x}_{k-1}^i,\xi_k^i)\|\|y_{k-1}^i\|+\|\nabla f_i(\hat{x}^i_k,\xi^i_k)\nonumber\\
          &\quad-(1-\gamma_k)\nabla f_i(\hat{x}_{k-1}^i,\xi_k^i)\|^2\nonumber\\
          &\overset{(a)}{\leqslant}\|\nabla f_i(\hat{x}_k^i,\xi_k^i)-\nabla f_i(\hat{x}_{k-1}^i,\xi_k^i)+\gamma_k\nabla f_i(\hat{x}_{k-1}^i,\xi_k^i)\|^2\nonumber\\
          &\quad+(1-\gamma_k)^2\|y^i_{k-1}\|^2+2(1-\gamma_k)\big(\|\nabla f_i(\hat{x}^i_k,\xi_k^i)\nonumber\\
          &\quad-\nabla f_i(\hat{x}_{k-1}^i,\xi_k^i)\|+\gamma_k\|\nabla f_i(\hat{x}_{k-1}^i,\xi_k^i)\|\big)\|y_{k-1}^i\|\nonumber\\
          &\overset{(b)}{\leqslant}\!2L^2\|\hat{x}_k^i\!-\!\hat{x}_{k-1}^i\|^2\!+\!2\gamma_k^2\|\nabla f_i(\hat{x}_{k-1}^i,\xi_k^i)\|^2\!+\!(1\!-\!\gamma_k)\|y_{k-1}^i\|^2\nonumber\\
          &\quad+2\big(L\|\hat{x}_k^i-\hat{x}_{k-1}^i\|+\gamma_k\|\nabla f_i(\hat{x}_{k-1}^i,\xi_k^i)\|\big)\|y_{k-1}^i\|\nonumber\\
          &\overset{(c)}{\leqslant}(1-\gamma_k)\|y_{k-1}^i\|^2+2\gamma_k^2\|\nabla f_i(\hat{x}_{k-1}^i,\xi_k^i)\|^2+2[L(D\nonumber\\
          &\quad+2C_1)\eta_{k-1}+\gamma_k\|\nabla f_i(\hat{x}_{k-1}^i,\xi_k^i)\|]\|y_{k-1}^i\|\nonumber\\
          &\quad+2L^2(D+2C_1)^2\eta_{k-1}^2
\end{align}
where $(a)$ is due to the triangle inequality; $(b)$ holds because of the fact $1-\gamma_k\leqslant 1$ and the $L$-smooth property of function $f_i$; $(c)$ follows from \eqref{lemma32equ}. Taking the conditional expectation of \eqref{qq1} on $\mathcal{F}_k$, we obtain
\begin{align}\label{eq23}
	\mathbb{E}_k[\|y_k^i\|^2]
	&\leqslant(1-\gamma_k)\|y_{k-1}^i\|^2+2L(D+2C_1)\eta_{k-1}\|y_{k-1}^i\|\nonumber\\
	&\quad+2\gamma_k\mathbb{E}_k[\|\nabla f_i(\hat{x}_{k-1}^i,\xi_k^i)\|]\|y_{k-1}^i\|\nonumber\\
	&\quad+2\gamma_k^2\mathbb{E}_k[\|\nabla f_i(\hat{x}_{k-1}^i,\xi_k^i)\|^2]\nonumber\\
	&\quad+2L^2(D+2C_1)^2\eta_{k-1}^2\nonumber\\
	&\leqslant(1\!-\!\gamma_k)\|y_{k\!-\!1}^i\|^2\!\!+\!2L(D\!+\!2C_1)\eta_{k\!-\!1}\|y_{k\!-\!1}^i\|\!+\!2\gamma_k^2G^2\nonumber\\
	&\quad+2\gamma_kG\|y_{k-1}^i\|+2L^2(D+2C_1)^2\eta_{k-1}^2
\end{align}
where the last inequality follows from Fact 2. Taking the full expectation of \eqref{eq23} and choosing the step sizes $\gamma_k=\frac{2}{k+1}$, $\eta_k=\frac{2}{k+2}$, we arrive at
\begin{align}
	\mathbb{E}[\|y_k^i\|^2]
	&\leqslant(1\!-\!\gamma_k)\mathbb{E}[\|y_{k-1}^i\|^2]\!+2L(D+2C_1)\eta_{k-1}\mathbb{E}[\|y_{k-1}^i\|]\nonumber\\
	&\quad+\!2\gamma_kG\mathbb{E}[\|y_{k-1}^i\|]\!+\!2\gamma_k^2G^2+2L^2(D\!+\!2C_1)^2\eta_{k-1}^2\nonumber\\
	&\leqslant(1\!-\!\gamma_k)\mathbb{E}[\|y_{k-1}^i\|^2]\!+\!2L(D\!+\!2C_1)\psi\eta_{k-1}+2\gamma_kG\psi\nonumber\\
	&\quad+2\gamma_k^2G^2+2L^2(D+2C_1)^2\eta_{k-1}^2\nonumber\\
	&\leqslant(1-\frac{2}{k+1})\mathbb{E}[\|y_{k-1}^i\|^2]+\frac{4L(D+2C_1)\psi+4G\psi}{k+1}\nonumber\\
	&\quad+\frac{8G^2+8L^2(D+2C_1)^2}{k+1}\nonumber
\end{align}
where the second inequality follows from \eqref{lemma44eqw1} of Lemma \ref{lemma4444}. By using Lemma \ref{lemmalemma1} ($r_1=1,~r_2=1,~A=2$ and $B=4L(D+2C_1)\psi+4G\psi+8G^2+8L^2(D+2C_1)^2$), we obtain
$
	\mathbb{E}[\|y_{k}^i\|^2]\leqslant\hat{\psi}
$,
where $\hat{\psi}=\mathop{\mathrm{max}}\limits_{i\in\mathcal{N}}\{\|y_1^i\|^2,4L(D+2C_1)\psi+4G\psi+8G^2+8L^2(D+2C_1)^2\}$.
\end{proof}

\subsection{Proof of Lemma \ref{lemma45}}\label{proofC}
\begin{proof}
	
	 It follows from the properties of Euclidean norm that $
		\mathbb{E}[\|p_k^i-\bar{y}_k\|^2]\leqslant\mathbb{E}[\sum_{i=1}^{n}\|p_k^i-\bar{y}_k\|^2].$
	Then, proving inequality $\mathbb{E}[\|p_k^i-\bar{y}_k\|^2]\leqslant 4C_2/(k+2)^2$ can be transformed into proving the inequality
	\begin{align}\label{lem2ex3}
		&\quad\mathbb{E}\bigg[\sum_{i=1}^{n}\|p_k^i-\bar{y}_k\|^2\bigg]\nonumber\\
		&\leqslant C_2\eta^2_k
		=\!\frac{k_0^3(4n)^{k_0\!+\!1}(12L^2(D\!+\!2C_1)^2\!+\!12(G^2\!+\!\hat{\psi}))}{(k+2)^2}
	\end{align}
by using induction on $k$. We first prove that \eqref{lem2ex3} holds for all $1\leqslant k\leqslant k_0-2$. According to \eqref{algeua1} and \eqref{algeua2} of Algorithm \ref{alg:1}, we have $p_k^i=s^i_{k+1}+y_k^i-y^i_{k+1}$. Therefore,
\begin{align}\label{qew1}
	&\quad\mathbb{E}\bigg[\sum_{i=1}^{n}\|p_k^i-\bar{y}_k\|^2\bigg]\nonumber\\
	&\leqslant\mathbb{E}\bigg[\sum_{i=1}^{n}\bigg(4\|s_{k+1}^i\|^2+4\|y_k^i\|^2+4\|y_{k+1}^i\|^2+\frac{4}{n^2}\bigg\|\sum_{j=1}^{n}{y}_k^j\bigg\|^2\bigg)\bigg]\nonumber\\
	&\leqslant 4\sum_{i=1}^{n}\mathbb{E}[\|s_{k+1}^i\|^2]+8n\hat{\psi}+4n\hat{\psi},
\end{align}
where the first inequality holds for the fact that $\|\sum_{i=1}^nz_i\|^2\leqslant n\sum_{i=1}^n\|z_i\|^2$ ($z_i$ is an arbitrary vector for $\forall i\in\{1,2,\cdots,n\}$) and the last inequality is due to \eqref{lemma45eqw1}. The first term of RHS of \eqref{qew1} can be written as
\begin{align}\label{qrt1}
  \mathbb{E}[\|s_{k+1}^i\|^2]&=\mathbb{E}\bigg[\bigg\|\sum_{j=1}^{n}c_{ij}s_k^j+y_{k+1}^i-y_k^i\bigg\|^2\bigg]\nonumber\\
  &\leqslant3\mathbb{E}\bigg[\bigg\|\sum_{j=1}^{n}c_{ij}s_k^j\bigg\|^2+\|y_{k+1}^i\|^2+\|y_k^i\|^2\bigg]\nonumber\\
  &\overset{(a)}{\leqslant}3\bigg(\mathbb{E}\bigg[\bigg\|\sum_{j=1}^{n}c_{ij}s_k^j\bigg\|^2\bigg]+2\hat{\psi}\bigg)\nonumber\\
  &\leqslant3\bigg(n\mathbb{E}\bigg[\sum_{j=1}^{n}\|c_{ij}s_k^j\|^2\bigg]+2\hat{\psi}\bigg)\nonumber\\
  &\overset{(b)}{\leqslant}3n\sum_{j=1}^{n}c_{ij}\mathbb{E}[\|s_k^j\|^2]+6\hat{\psi}\nonumber\\
  &\overset{(c)}{\leqslant}(3n)^kG^2+6k(3n)^{(k-1)}\hat{\psi}
\end{align}
where $(a)$ follows from \eqref{lemma45eqw1};
$(b)$ is due to the fact that $\|\sum_{i=1}^nz_i\|^2\leqslant n\sum_{i=1}^n\|z_i\|^2$ ($z_i$ is an arbitrary vector for $\forall i\in\{1,2,\cdots,n\}$) and the fact that $c_{ij}\leqslant 1$ for all $i,j\in\mathcal{N}$; $(c)$ is because $\mathbb{E}[\|s_1^j\|^2]=\mathbb{E}[\|\nabla f_j(\hat{x}_1^j,\xi_1^j)\|^2]\leqslant G^2$. Now substituting \eqref{qrt1} into \eqref{qew1}, we arrive at
\begin{align}
	&\quad\mathbb{E}\bigg[\sum_{i=1}^{n}\|p_k^i-\bar{y}_k\|^2\bigg]\nonumber\\
	&\leqslant4n(3n)^kG^2+24nk(3n)^{(k-1)}\hat{\psi}+8n\hat{\psi}+4n\hat{\psi}\nonumber\\
	&\leqslant(4n)^{k+1}G^2+6k(4n)^{k}\hat{\psi}+12n\hat{\psi},\nonumber
\end{align}
that is, $\mathbb{E}[\sum_{i=1}^{n}\|p_k^i-\bar{y}_k\|^2]\leqslant(4n)^{k_0-1}G^2+6(k_0-2)(4n)^{k_0-2}\hat{\psi}+12n\hat{\psi}$ holds for all $1\leqslant  k\leqslant k_0-2$.
Hence, the inequality \eqref{lem2ex3} is obviously true for $k=1$ to $k=k_0-2$.

 For induction step, we assume that \eqref{lem2ex3} holds for some $k\geqslant k_0-2$. Define the slack variable $\Delta y_{k+1}^i:=y_{k+1}^i-y_{k}^i$. Then, we observe that $s_{k+1}^i=\Delta y_{k+1}^i+p_{k}^i$ due to the definition of $s_{k+1}^i$ and $p_{k+1}^i$. From \eqref{fact1} and Lemma \ref{lemma0} (a), we have
\begin{align}\label{lem2ex4}
&\quad\mathbb{E}\bigg[\sum_{i=1}^{n}\|p_{k+1}^i-\bar{y}_{k+1}\|^2\bigg]\nonumber\\
	&\leqslant\mathbb{E}\bigg[|\lambda|^2\sum_{i=1}^{n}\|\Delta y_{k+1}^i+p_k^i-\bar{y}_{k+1}\|^2\bigg].
\end{align}
Define $\Delta Y_{k+1}:=\bar{y}_{k+1}-\bar{y}_k$. Then,
\begin{align}\label{213}
	&\quad\sum_{i=1}^{n}\|\Delta y_{k+1}^i+p_k^i-\bar{y}_{k+1}\|^2\nonumber\\
	&=\sum_{i=1}^{n}\|p_k^i-\bar{y}_k+\Delta y_{k+1}^i-\Delta Y_{k+1}\|^2\nonumber\\
	&\leqslant\sum_{i=1}^{n}\big[\|p_k^i-\bar{y}_k\|^2+\|\triangle y_{k+1}^i-\triangle Y_{k+1}\|^2\nonumber\\
	&\quad+2\|\triangle y_{k+1}^i-\triangle Y_{k+1}\|\|p_k^i-\bar{y}_k\|\big].
\end{align}
According to the definition of $\triangle y_{k+1}^i$, we get
\begin{align}\label{lem2ex5}
	&\quad\mathbb{E}[\|\Delta y_{k+1}^i\|^2]
	=\mathbb{E}[\|y_{k+1}^i-y_k^i\|^2]\nonumber\\
	&=\mathbb{E}[\|\nabla f_i(\hat{x}_{k+1}^i,\xi_{k+1}^i)-\nabla f_i(\hat{x}_k^i,\xi_{k+1}^i)\nonumber\\
	&\quad+\gamma_{k+1}(\nabla f_i(\hat{x}_k^i,\xi_{k+1}^i)-y_k^i)\|^2]\nonumber\\
	&\overset{(a)}{\leqslant}3\mathbb{E}[\|\nabla f_i(\hat{x}_{k+1}^i,\xi_{k+1}^i)-\nabla f_i(\hat{x}_k^i,\xi_{k+1}^i)\|^2]\nonumber\\
	&\quad+3\gamma_{k+1}^2\mathbb{E}[\|\nabla f_i(\hat{x}_k^i,\xi_{k+1}^i)\|^2]+3\gamma_{k+1}^2\mathbb{E}[\|y_k^i\|^2]\nonumber\\
	&\overset{(b)}{\leqslant}3L^2\mathbb{E}[\|\hat{x}_{k+1}^i-\hat{x}_k^i\|^2]+3\gamma_{k+1}^2(G^2+\hat{\psi})\nonumber\\
	&\overset{(c)}{\leqslant}3L^2(D+2C_1)^2\eta_k^2+3(G^2+\hat{\psi})\gamma_{k+1}^2,
\end{align}
where $(a)$ is due to the fact that $\|\sum_{i=1}^nz_i\|^2\leqslant n\sum_{i=1}^n\|z_i\|^2$ ($z_i$ is an arbitrary vector for $\forall i\in\{1,2,\cdots,n\}$); $(b)$ holds  because of the $L$-smooth property of function $f_i(x,\xi)$, Fact 3 and \eqref{lemma45eqw1}; $(c)$ follows from \eqref{lemma32equ}. It follows from the definition of $\Delta Y_{k+1}$ that
\begin{align}\label{lem2ex9}
	&\quad\mathbb{E}[\|\Delta y_{k+1}^i-\Delta Y_{k+1}\|^2]\nonumber\\
	&=\mathbb{E}\bigg[\bigg\|\Delta y_{k+1}^i-\frac{1}{n}\sum_{i=1}^{n}\Delta y_{k+1}^i\bigg\|^2\bigg]\nonumber\\
	&=\mathbb{E}\bigg[\bigg\|(1-\frac{1}{n})\Delta y_{k+1}^i-\frac{1}{n}\sum_{j\neq i}\Delta y_{k+1}^j\bigg\|^2\bigg]\nonumber\\
	&\leqslant2(1-\frac{1}{n})\mathbb{E}[\|\Delta y_{k+1}^i\|^2]+\frac{2}{n}\sum_{j\neq i}\mathbb{E}[\|\Delta y_{k+1}^j\|^2]\nonumber\\
	&\overset{(a)}{\leqslant}4(1-\frac{1}{n})[3L^2(D+2C_1)^2\eta_k^2+3(G^2+\hat{\psi})\gamma_{k+1}^2]\nonumber\\
	&\overset{(b)}{\leqslant}[12L^2(D+2C_1)^2+12(G^2+\hat{\psi})]\eta_k^2,
\end{align}
where $(a)$ follows from \eqref{lem2ex5};
$(b)$ is by the fact $1-\frac{1}{n}\leqslant1$.
Substituting \eqref{lem2ex9} into \eqref{213} and taking the full expectation of \eqref{213}, we obtain
\begin{align}\label{ert41}
	&\quad\mathbb{E}\bigg[\sum_{i=1}^{n}\|\triangle y_{k+1}^i+p_{k}^i-\bar{y}_{k+1}\|^2\bigg]\nonumber\\
	&\overset{(a)}{\leqslant}\mathbb{E}\bigg[\sum_{i=1}^{n}\|p_{k}^i-\bar{y}_k\|^2\bigg]+n[12L^2(D+2C_1)^2+12(G^2\nonumber\\
	&\quad+\hat{\psi})]\eta_k^2+2\sum_{i=1}^{n}(\mathbb{E}[\|\triangle y_{k+1}^i-\triangle Y_{k+1}\|^2])^\frac{1}{2}\nonumber\\
	&\quad\times(\mathbb{E}[\|p_{k}^i-\bar{y}_k\|^2])^\frac{1}{2}\nonumber\\
	&\overset{(b)}{\leqslant}C_2\eta^2_k+n[12L^2(D+2C_1)^2+12(G^2+\hat{\psi})]\eta_k^2\nonumber\\
	&\quad+2n\eta^2_k\sqrt{[12L^2(D+2C_1)^2+12(G^2+\hat{\psi})]C_2}\nonumber\\
	&\overset{(c)}{\leqslant}\eta^2_k\bigg(C_2+n^2[12L^2(D+2C_1)^2+12(G^2+\hat{\psi})]\nonumber\\
	&\quad+2n\sqrt{[12L^2(D+2C_1)^2+12(G^2+\hat{\psi})]C_2}\bigg)\nonumber\\
	&\leqslant\eta^2_k\bigg(\sqrt{C_2}+\frac{\sqrt{C_2}}{k_0}\bigg)^2\nonumber\\
	&=\bigg(\frac{k_0+1}{k_0}\sqrt{C_2}\eta_k\bigg)^2,
\end{align}
where $(a)$ is due to the H$\ddot{o}$lder's inequality ($\mathbb{E}[|XY|]\leqslant(\mathbb{E}[|X|^2])^{\frac{1}{2}}(\mathbb{E}[|Y|^2])^{\frac{1}{2}}$); $(b)$ follows from \eqref{lem2ex9} and the induction hypothesis \eqref{lem2ex3}; $(c)$ is due to the fact $n\geqslant 1$.
Substituting \eqref{ert41}, $|\lambda|\leqslant\big(\frac{k_0}{k_0+1}\big)^2$ and $\eta_k=\frac{2}{k+2}$ into \eqref{lem2ex4}, we obtain
	\begin{align}
	\mathbb{E}\bigg[\sum_{i=1}^{n}\|p_{k+1}^i-\bar{y}_{k+1}\|^2\bigg]
		&\leqslant\bigg(\frac{2k_0}{(k+2)(k_0+1)}\sqrt{C_2}\bigg)^2\nonumber\\
		&\leqslant\bigg(\frac{2(k+2)}{(k\!+\!3)(k\!+\!2)}\sqrt{C_2}\bigg)^2
		\!=C_2\eta^2_{k+1},\nonumber
	\end{align}
	where the last inequality is because of  the  monotonically increasing property of function $g(v)=v/(1+v)$ with respect to $v$ over $[0,\infty)$. Hence, \eqref{lem2ex3} holds for iteration $k+1$. Therefore, $\mathbb{E}[\|p_k^i-\bar{y}_k\|^2]\leqslant 4C_2/(k+2)^2$ holds for all $k\geqslant 1$.
\end{proof}

\subsection{Proof of Lemma \ref{lemmalemma34}}\label{proofD}
\begin{proof}
	(a) It follows from the definition of $\bar{y}_k$ that
	\begin{align}\label{lemma3ma1}
		\bar{P}_k-\bar{y}_k
	&	=\bar{P}_k-\frac{1}{n}\sum_{i=1}^{n} y_{k}^i\nonumber\\
	&	=\bar{P}_k-(1-\gamma_k)\bar{y}_{k-1}-\frac{1}{n}\sum_{i=1}^{n}\nabla f_i(\hat{x}_{k}^i,\xi_{k}^i)\nonumber
		\\
		&\quad +\frac{1-\gamma_k}{n}\sum_{i=1}^{n}\nabla f_i(\hat{x}_{k-1}^i,\xi_{k}^i).
	\end{align}
	Introducing $(1-\gamma_k)\bar{P}_{k-1}$ into \eqref{lemma3ma1} and taking norm square of \eqref{lemma3ma1}, we arrive at
\begin{align}\label{lemma3ma2}
	\|\bar{P}_k-\bar{y}_k\|^2
	&=\big\|(1\!-\!\gamma_k)(\bar{P}_{k-1}\!-\!\bar{y}_{k-1})\!-\!\frac{1}{n}\!\sum_{i=1}^{n}\!\nabla f_i(\hat{x}_{k}^i,\xi_{k}^i)\!+\!\bar{P}_{k}\nonumber\\
	&\quad+\!(1\!-\!\gamma_k)(\frac{1}{n}\!\sum_{i=1}^{n}\!\nabla f_i(\hat{x}_{k\!-\!1}^i,\!\xi_{k}^i)\!-\!\bar{P}_{k\!-\!1})\big\|^2\!.
\end{align}
Taking the conditional expectation of \eqref{lemma3ma2} on $\mathcal{F}_k$, we obtain
\begin{align}\label{lemma3ma3}
	&\quad\mathbb{E}_k[\|\bar{P}_k-\bar{y}_k\|^2]\nonumber\\
	&=(1\!-\!\gamma_k)^2\|\bar{P}_{k-1}\!-\!\bar{y}_{k-1}\|^2\!+\!2(1-\gamma_k)(\bar{P}_{k-1}\!-\!\bar{y}_{k-1})^{\mathrm{T}}\!\bigg(\!\mathbb{E}_k\bigg[\bar{P}_k\nonumber\\
	&\quad-\frac{1}{n}\sum_{i=1}^{n}\nabla f_i(\hat{x}_{k}^i,\xi_{k}^i)\bigg]+(1-\gamma_k)\mathbb{E}_k\bigg[\frac{1}{n}\sum_{i=1}^{n}\nabla f_i(\hat{x}_{k-1}^i,\xi_{k}^i)\nonumber\\
	&\quad-\bar{P}_{k-1}\bigg]\bigg)+\mathbb{E}_k\bigg[\bigg\|\bar{P}_k-\frac{1}{n}\sum_{i=1}^{n}\nabla f_i(\hat{x}^i_k,\xi_{k}^i)+(1\nonumber\\
	&\quad-\gamma_k)\bigg(\frac{1}{n}\sum_{i=1}^{n}\nabla f_i(\hat{x}_{k-1}^i,\xi_{k}^i)-\bar{P}_{k-1}\bigg)\bigg\|^2\bigg]\nonumber\\
	&=(1-\gamma_k)^2\|\bar{P}_{k-1}-\bar{y}_{k-1}\|^2+\mathbb{E}_k\bigg[\bigg\|\bar{P}_k-\frac{1}{n}\sum_{i=1}^{n}\nabla f_i(\hat{x}_{k}^i,\xi_{k}^i)\nonumber\\
	&\quad+(1-\gamma_k)\bigg(\frac{1}{n}\sum_{i=1}^{n}\nabla f_i(\hat{x}_{k-1}^i,\xi_{k}^i)-\bar{P}_{k-1}\bigg)\bigg\|^2\bigg],
\end{align}
where the last equality is because $\mathbb{E}_k[\bar{P}_k-\frac{1}{n}\sum_{i=1}^{n}\nabla f_i(\hat{x}_{k}^i,\xi_{k}^i)]=\mathbb{E}_k[\frac{1}{n}\sum_{i=1}^{n}\nabla F_i(\hat{x}_{k}^i)-\frac{1}{n}\sum_{i=1}^{n}\nabla f_i(\hat{x}_{k}^i,\xi_{k}^i)]=\frac{1}{n}\sum_{i=1}^{n}\mathbb{E}_k[\nabla F_i(\hat{x}_{k}^i)-\nabla f_i(\hat{x}_{k}^i,\xi_{k}^i)]=0$. Adding and subtracting $\gamma_k(\bar{P}_k-\frac{1}{n}\sum_{i=1}^n\nabla f_i(\hat{x}_{k}^i,\xi_{k}^i))$ in the second term of RHS of \eqref{lemma3ma3}, we have
\begin{align}\label{lemma3ma4}
	&\quad\mathbb{E}_k\bigg[\bigg\|(1-\gamma_k)\bigg(\frac{1}{n}\sum_{i=1}^n\nabla f_i(\hat{x}_{k-1}^i,\xi_{k}^i)-\bar{P}_{k-1}+\bar{P}_k\nonumber\\
	&\quad-\!\frac{1}{n}\!\sum_{i=1}^n\nabla\! f_i(\hat{x}_{k}^i,\xi_{k}^i)\bigg)\!+\!\gamma_k\bigg(\bar{P}_k\!-\!\frac{1}{n}\!\sum_{i=1}^n\!\nabla\! f_i(\hat{x}_{k}^i,\xi_{k}^i)\bigg)\bigg\|^2\bigg]\nonumber\\
	&\overset{(a)}{\leqslant}\frac{3(1\!-\!\gamma_k)^2}{n}\!\!\sum_{i=1}^n\!\mathbb{E}_k[\|\nabla f_i(\hat{x}_{k-1}^i,\!\xi_{k}^i)\!-\!\nabla f_i(\hat{x}_{k}^i,\!\xi_{k}^i)\|^2]\nonumber\\
	&\quad+3(1-\gamma_k)^2\bigg\|\frac{1}{n}\sum_{i=1}^{n}(\nabla F_i(\hat{x}^i_k)-\nabla F_i(\hat{x}^i_{k-1}))\bigg\|^2\nonumber\\
	&\quad+3\gamma_k^2\mathbb{E}_k\bigg[\bigg\|\frac{1}{n}\sum_{i=1}^{n}\nabla f_i(\hat{x}_{k}^i,\xi_{k}^i)-\bar{P}_k\bigg\|^2\bigg]\nonumber\\
	&\overset{(b)}{\leqslant}\frac{3(1-\gamma_k)^2}{n}\sum_{i=1}^n\mathbb{E}_k[L^2\|\hat{x}_{k}^i-\hat{x}_{k-1}^i\|^2]+3\gamma_k^2\delta^2\nonumber\\
	&\quad+\frac{3(1-\gamma_k)^2L^2}{n}\sum_{i=1}^{n}\|\hat{x}_{k}^i-\hat{x}_{k-1}^i\|^2\nonumber\\
	&\overset{(c)}{\leqslant}6L^2(D+2C_1)^2\eta_{k-1}^2+3\gamma_k^2\delta^2,
\end{align}
where $(a)$ is due to the fact that $\|\sum_{i=1}^nz_i\|^2\leqslant n\sum_{i=1}^n\|z_i\|^2$ ($z_i$ is an arbitrary vector for $\forall i\in\{1,2,\cdots,n\}$) and $\bar{P}_k=\frac{1}{n}\sum_{i=1}^{n}\nabla F_i(\hat{x}_k^i)$;
$(b)$ is due to Assumption \ref{assup5}, the $L$-smooth property of function $f_i$ and $F_i$; $(c)$ is obtained by  \eqref{lemma32equ} and the fact $(1-\gamma_k)^2\leqslant1$. Substituting \eqref{lemma3ma4} into \eqref{lemma3ma3}, we obtain
\begin{align}\label{lemma3ma5}
\mathbb{E}_k[\|\bar{P}_k-\bar{y}_k\|^2]&\leqslant	(1-\gamma_k)\|\bar{P}_{k-1}-\bar{y}_{k-1}\|^2\nonumber\\
&\quad+6L^2(D+2C_1)^2\eta_{k-1}^2+3\gamma_k^2\delta^2.
\end{align}

(b) It is obvious that $\mathbb{E}[\|\bar{P}_{1}-\bar{y}_1\|^2]=0\leqslant\frac{C_3}{k+2}$, that is, \eqref{lemma44eq1} holds when $k=1$. Now, let's discuss the case when $k\geqslant2$. Taking full expectation of \eqref{lemma3ma5} and choosing the step sizes $\gamma_k=\frac{2}{k+1}$, $\eta_k=\frac{2}{k+2}$, we have
\begin{align}
	&\quad\mathbb{E}[\|\bar{P}_k-\bar{y}_k\|^2]\nonumber\\
	&\leqslant(1-\frac{2}{k+1})\mathbb{E}[\|\bar{P}_{k-1}-\bar{y}_{k-1}\|^2]+\frac{24L^2(D+2C_1)^2+12\delta^2}{(k+1)^2}.\nonumber
\end{align}
By using Lemma \ref{lemmalemma1} ($r_1=1,~r_2=2,~A=2$ and $B=24L^2(D+2C_1)^2+12\delta^2$), we obtain $
	\mathbb{E}[\|\bar{P}_k-\bar{y}_k\|^2]\leqslant\frac{C_3}{k+2},$
where $C_3=24L^2(D+2C_1)^2+12\delta^2$.
\end{proof}

\subsection{Proof of Lemma \ref{lem1}}\label{proofF1}
\begin{proof}
	Adding and subtracting $\bar{P}_k$ and $\bar{y}_k$ to $\|\nabla F(\bar{x}_k)-p_k^i\|^2$, we have
	\begin{align}\label{qw1}
		&\quad\mathbb{E}[\|\nabla F(\bar{x}_k)-p_k^i\|^2]\nonumber\\
		&=\mathbb{E}[\|\nabla F(\bar{x}_k)-\bar{P}_k+\bar{P}_k-\bar{y}_k+\bar{y}(k)-p_k^i\|^2]\nonumber\\
		&\leqslant3\mathbb{E}[\|\nabla F(\bar{x}_k)-\bar{P}_k\|^2]+3\mathbb{E}[\|\bar{P}_k-\bar{y}_k\|^2]\nonumber\\
		&\quad+3\mathbb{E}[\|\bar{y}_k-p_k^i\|^2],
	\end{align}
	where the inequality is due to the fact $\|\sum_{i=1}^nz_i\|^2\leqslant n\sum_{i=1}^n\|z_i\|^2$, $z_i(i\in\{1,2,\cdots,n\})$ is an arbitrary vector. The first term of RHS of \eqref{qw1} can be written as
	\begin{align}\label{qw2}
		&\quad3\mathbb{E}[\|\nabla F(\bar{x}_k)-\bar{P}_k\|^2]\nonumber\\
		&=3\mathbb{E}\bigg[\bigg\|\frac{1}{n}\sum_{i=1}^n\nabla F_i(\bar{x}_k)-\frac{1}{n}\sum_{i=1}^n\nabla F_i(\hat{x}_{k}^i)\bigg\|^2\bigg]\nonumber\\
		&\leqslant3\mathbb{E}\bigg[\bigg(\frac{1}{n}\sum_{i=1}^n\|\nabla F_i(\bar{x}_{k}^i)-\nabla F_i(\hat{x}_{k}^i)\|\bigg)^2\bigg]\nonumber\\
		&\overset{(a)}{\leqslant}3\frac{L^2}{n}\sum_{i=1}^n\mathbb{E}[\|\bar{x}_k-\hat{x}_{k}^i\|^2]		\overset{(b)}{\leqslant} 3L^2C^2_1\eta_k^2,
	\end{align}
	where $(a)$ is due to the $L$-smooth property of function $F_i$; $(b)$ follows form \eqref{lem1equ}. Substituting \eqref{qw2}, \eqref{lemma44eq1} and \eqref{lemma45eq1} into \eqref{qw1}, we obtain
	\begin{align}
		\mathbb{E}[\|\nabla F(\bar{x}_k)-p_k^i\|^2]
		&\leqslant\frac{12L^2C^2_1}{(k+2)^2}+\frac{3C_3}{k+2}+ \frac{12C_2}{(k+2)^2}\nonumber\\
		&\leqslant\frac{12L^2C^2_1+3C_3+12C_2}{k+2}.\nonumber
	\end{align}
	\end{proof}

\subsection{Proof of Theorem \ref{the1}}\label{proofH}
\begin{proof}
	Since the function $F(x)$ is $L$-smooth, we have
	\begin{align}\label{equ1}
		F(\bar{x}_{k+1})&\leqslant F(\bar{x}_k)+\nabla^\mathrm{T} F(\bar{x}_k)(\bar{x}_{k+1}\!-\bar{x}_k)\!+\!\frac{L}{2}\|\bar{x}_{k+1}-\bar{x}_k\|^2\nonumber\\
		&=F(\bar{x}_k)+\eta_k\nabla^\mathrm{T} F(\bar{x}_k)(\bar{\theta}_k\!-\!\bar{x}_k)\!+\!\frac{L\eta^2_k}{2}\|\bar{\theta}_k\!-\!\bar{x}_k\|^2\nonumber\\
		&\leqslant F(\bar{x}_k)\!+\eta_k\nabla^\mathrm{T} F(\bar{x}_k)(\bar{\theta}_k-\bar{x}_k)+\frac{L\eta^2_k}{2}D^2,
	\end{align}
where the first equation is due to Lemma \ref{lemma0} (b) and the last inequality follows from Assumption \ref{assumption3}. From the definition of $\bar{\theta}_k$, the second term of RHS of \eqref{equ1} can be written as
\begin{align}\label{equ2}
	&\quad\nabla^\mathrm{T} F(\bar{x}_k)(\bar{\theta}_k-\bar{x}_k)
	=\frac{1}{n}\sum_{i=1}^{n}\nabla^\mathrm{T} F(\bar{x}_k)(\theta_k^i-\bar{x}_k)\nonumber\\
	&=\frac{1}{n}\sum_{i=1}^{n}(\nabla F(\bar{x}_k)-p_k^i+p_k^i)^\mathrm{T}(\theta_k^i-\bar{x}_k)\nonumber\\
	&\overset{(a)}{\leqslant}\frac{1}{n}\sum_{i=1}^{n}\big[(\nabla F(\bar{x}_k)-p_k^i)^\mathrm{T}(\theta_k^i-\bar{x}_k)+{p_k^i}^\mathrm{T}(x^*-\bar{x}_k)\big]\nonumber\\
	&=\frac{1}{n}\sum_{i=1}^{n}\big[(\nabla F(\bar{x}_k)-p_k^i)^\mathrm{T}(\theta_k^i-x^*+x^*-\bar{x}_k)\nonumber\\
	&\quad+{p_k^i}^\mathrm{T}(x^*-\bar{x}_k)\big]\nonumber\\
	&=\frac{1}{n}\!\sum_{i=1}^{n}\!\!\big[(\nabla F(\bar{x}_k)\!-\!p_k^i)^\mathrm{T}\!(\theta_k^i\!-\!x^*)\!+\!\nabla^\mathrm{T} F(\bar{x}_k)\!(x^*\!-\!\bar{x}_k)\big]\nonumber\\
	&\overset{(b)}{\leqslant}\frac{1}{n}\sum_{i=1}^{n}\|\nabla F(\bar{x}_k)-p_k^i\|\|\theta_k^i\!-\!x^*\|+F(x^*)\!-\!F(\bar{x}_k)\nonumber\\
	&\leqslant\frac{1}{n}\sum_{i=1}^{n}\|\nabla F(\bar{x}_k)-p_k^i\|D+F(x^*)-F(\bar{x}_k),
\end{align}
where $(a)$ follows from the optimality of $\theta_i$ in \eqref{algstep4}; $(b)$ is due to the property of convex function $F$, i.e., $F(x^*)-F(\bar{x}_k)\geqslant\nabla^\mathrm{T} F(\bar{x}_k)(x^*-\bar{x}_k)$. Substituting \eqref{equ2} into \eqref{equ1}, we arrive at
\begin{align}\label{equ3}
	F(\bar{x}_{k+1})&\leqslant F(\bar{x}_k)+\eta_k\bigg(\frac{1}{n}\sum_{i=1}^{n}\|\nabla F(\bar{x}_k)-p_k^i\|D\nonumber\\
	&\quad+F(x^*)-F(\bar{x}_k)\bigg)+\frac{L\eta^2_k}{2}D^2\nonumber\\
	&=(1\!-\!\eta_k)F(\bar{x}_k)\!+\!\frac{\eta_k}{n}\sum_{i=1}^{n}\|\nabla F(\bar{x}_k)\!-\!p_k^i\|D\nonumber\\
	&\quad+\eta_kF(x^*)+\frac{L\eta^2_k}{2}D^2.
\end{align}
Subtracting $F(x^*)$ from both sides of inequality \eqref{equ3}, we arrive at
\begin{align}\label{equ4}
	&\quad F(\bar{x}_{k+1})-F(x^*)\nonumber\\
	&\leqslant(1-\eta_k)(F(\bar{x}_k)-F(x^*))+\frac{L\eta^2_k}{2}D^2\nonumber\\
	&\quad+\frac{\eta_k}{n}\sum_{i=1}^{n}\|\nabla F(\bar{x}_k)-p_k^i\|D.
\end{align}
Taking the expectation and using the Jensen's inequality on the last term of \eqref{equ4}, that is, $\mathbb{E}[\|\nabla F(\bar{x}_k)-p_k^i\|]=\sqrt{(\mathbb{E}[\|\nabla F(\bar{x}_k)-p_k^i\|])^2}\leqslant\sqrt{\mathbb{E}[\|\nabla F(\bar{x}_k)-p_k^i\|^2]}$, we get
\begin{align}\label{equ5}
	&\quad\mathbb{E}[F(\bar{x}_{k+1})]-F(x^*)\nonumber\\
	&\leqslant(1-\eta_k)(\mathbb{E}[F(\bar{x}_k)]-F(x^*))+\frac{L\eta^2_k}{2}D^2\nonumber\\
	&\quad+\frac{\eta_k}{n}\sum_{i=1}^{n}\mathbb{E}[\|\nabla F(\bar{x}_k)-p_k^i\|]D\nonumber\\
	&\leqslant(1-\eta_k)(\mathbb{E}[F(\bar{x}_k)]-F(x^*))+\frac{L\eta^2_k}{2}D^2\nonumber\\
	&\quad+\frac{\eta_k}{n}\sum_{i=1}^{n}D\sqrt{\mathbb{E}[\|\nabla F(\bar{x}_k)-p_k^i\|^2]}.
\end{align}
Substituting $\gamma_k=\frac{2}{k+1}$, $\eta_k=\frac{2}{k+2}$ and \eqref{lemmaequal1} into \eqref{equ5}, we have
\begin{align}
	&\quad\mathbb{E}[F(\bar{x}_{k+1})]-F(x^*)\nonumber\\
	&\leqslant(1-\frac{2}{k+1})(\mathbb{E}[F(\bar{x}_k)]-F(x^*))+\frac{2LD^2}{(k+2)^2}\nonumber\\
	&\quad+\frac{2D\sqrt{12L^2C^2_1+3C_3+12C_2}}{(k+2)(k+2)^{\frac{1}{2}}}\nonumber\\
	&\leqslant(1-\frac{2}{k+2})(\mathbb{E}[F(\bar{x}_k)]-F(x^*))\nonumber\\
	&\quad+\frac{2LD^2+2D\sqrt{12L^2C^2_1+3C_3+12C_2}}{(k+2)^{\frac{3}{2}}}.
\end{align}
By using Lemma \ref{lemmalemma1} ($r_1=1,~r_2=\frac{3}{2},~A=2$ and $B=2LD^2+2D\sqrt{12L^2C^2_1+3C_3+12C_2}$), we obtain $
	\mathbb{E}[F(\bar{x}_{k+1})]-F(x^*)
	\leqslant\frac{C_4}{(k+3)^{\frac{1}{2}}}$, where $C_4=\mathrm{max}\{\sqrt{3}(F(\bar{x}_1)-F(x^*)),2LD^2+2D\sqrt{12L^2C^2_1+3C_3+12C_2}\}$.
\end{proof}

\subsection{Proof of Theorem \ref{the2}}\label{proofI}
\begin{proof}
	It follows from the definition of FW-gap (\ref{zx1}) that
	\begin{align}\label{mq1}
		g_k=\mathop{\mathrm{max}}\limits_{x\in\mathcal{X}}\langle\nabla F(\bar{x}_k),\bar{x}_k-x\rangle
		=\langle\nabla F(\bar{x}_k),\bar{x}_k-\hat{v}_k\rangle,
	\end{align}
	where
	\begin{align}\label{dfg1}
		\hat{v}(k)\in\mathop{\mathrm{argmin}}\limits_{v\in\mathcal{X}}\langle v,\nabla F(\bar{x}_k)\rangle.
	\end{align}
 According to the $L$-smooth property of $F$, we can write
	\begin{align}\label{the201}
		F(\bar{x}_{k+1})&\leqslant F(\bar{x}_k)+\nabla^\mathrm{T} F(\bar{x}_k)(\bar{x}_{k+1}\!-\!\bar{x}_k)\!+\!\frac{L}{2}\|\bar{x}_{k+1}-\bar{x}_k\|^2\nonumber\\
		&\overset{(a)}{=}F(\bar{x}_k)\!+\!\frac{\eta_k}{n}\sum_{i=1}^{n}(\nabla F(\bar{x}_k)+p_k^i-p_k^i)^\mathrm{T}(\theta_k^i-\bar{x}_k)\nonumber\\
		&\quad+\frac{L\eta^2_k}{2}\|\bar{\theta}_k-\bar{x}_k\|^2\nonumber\\
		&\overset{(b)}{\leqslant}F(\bar{x}_k)\!+\!\frac{\eta_k}{n}\sum_{i=1}^{n}\!\big[(p_k^i\!-\!\nabla F(\bar{x}_k)\!+\!\nabla\! F(\bar{x}_k))^\mathrm{T}\!(\hat{v}_k\!-\!\bar{x}_k)\nonumber\\
		&\quad+(\nabla F(\bar{x}_k)-p_k^i)^\mathrm{T}(\theta_k^i-\bar{x}_k)\big]+\frac{L\eta^2_k}{2}\|\bar{\theta}_k-\bar{x}_k\|^2\nonumber\\
		&\overset{(c)}{\leqslant}F(\bar{x}_k)+\frac{2\eta_k}{n}\sum_{i=1}^{n}\|\nabla F(\bar{x}_k)-p_k^i\|D\nonumber\\
		&\quad-\eta_kg_k+\frac{LD^2}{2}\eta^2_k,
	\end{align}
	where $(a)$ is because of Lemma \ref{lemma0} (b); $(b)$ is due to the fact $\theta_k^i=\mathrm{argmin}_{\phi\in\mathcal{X}}\langle p^i_k,\phi\rangle$ in \eqref{algstep4} and $\hat{v}(k)$ is defined in \eqref{dfg1}; $(c)$ follows from \eqref{mq1} and the Assumption \ref{assumption3}. Taking the expectation on both sides of \eqref{the201}, we arrive at
	\begin{align}\label{the202}
		\mathbb{E}[F(\bar{x}_{k+1})]
		&\leqslant\mathbb{E}[F(\bar{x}_k)]+\frac{2\eta_k}{n}\sum_{i=1}^{n}\mathbb{E}[\|\nabla F(\bar{x}_k)-p_k^i\|]D\nonumber\\
		&\quad-\eta_k\mathbb{E}[g_k]+\frac{LD^2}{2}\eta^2_k\nonumber\\
		&\overset{(a)}{\leqslant}\mathbb{E}[F(\bar{x}_k)]+\frac{2\eta_k}{n}\sum_{i=1}^{n}D\sqrt{\mathbb{E}[\|\nabla F(\bar{x}_k)-p_k^i\|^2]}\nonumber\\
		&\quad-\eta_k\mathbb{E}[g_k]+\frac{LD^2}{2}\eta^2_k\nonumber\\
		&\overset{(b)}{\leqslant}\mathbb{E}[F(\bar{x}_k)]+\frac{2\eta_kD\sqrt{12L^2C^2_1+3C_3+12C_2}}{(k+2)^{\frac{1}{2}}}\nonumber\\
		&\quad-\eta_k\mathbb{E}[g_k]+\frac{LD^2}{2}\eta^2_k,
	\end{align}
	where $(a)$ is due to the Jensen's inequality; $(b)$ follows from \eqref{lemmaequal1}. Summing from $k=1$ to $k=K$ on both sides of \eqref{the202}, we get
	\begin{align}\label{the203}
		\mathbb{E}\bigg[\sum_{k=1}^{K}\eta_kg_k\bigg]
		&\leqslant\mathbb{E}\bigg[\sum_{k=1}^{K}\big(F(\bar{x}_k)\!-\!F(\bar{x}_{k+1})\big)\bigg]\!+\!\sum_{k=1}^{K}\bigg(\frac{LD^2}{2}\eta^2_k\nonumber\\
		&\quad+\frac{2\eta_kD\sqrt{12L^2C^2_1+3C_3+12C_2}}{(k+2)^{\frac{1}{2}}}\bigg)\nonumber\\
		&=\mathbb{E}\big[F(\bar{x}_1)-F(\bar{x}_{K+1})\big]+\sum_{k=1}^{K}\bigg(\frac{LD^2}{2}\eta^2_k\nonumber\\
		&\quad+\frac{2\eta_kD\sqrt{12L^2C^2_1+3C_3+12C_2}}{(k+2)^{\frac{1}{2}}}\bigg)\nonumber\\
		&\leqslant\mathbb{E}\big[F(\bar{x}_1)-F(x^*)\big]+\sum_{k=1}^{K}\bigg(\frac{LD^2}{2}\eta^2_k\nonumber\\
		&\quad+\frac{2\eta_kD\sqrt{12L^2C^2_1+3C_3+12C_2}}{(k+2)^{\frac{1}{2}}}\bigg),
	\end{align}
	where the last inequality arises from the optimality of $x^*$. According to the property of $p$ series, we have $m-1\leqslant \sum_{k=1}^{2^m}\frac{2}{k+2},~\sum_{k=1}^{2^m}(\frac{2}{k+2})^2\leqslant8$, and there exists $\beta\in\mathbb{R}$ such that $\sum_{k=1}^{2^m}\frac{2}{(k+2)^{1.5}}\leqslant \beta.$	Define $g_a:=\mathop{\mathrm{min}}\limits_{k\in[1,K]}g_k$. Substituting $K=2^m$ and $\eta_k=\frac{2}{k+2}$  into \eqref{the203}, we arrive at
	\begin{align}
		&\quad\mathbb{E}[(m-1)g_a]\leqslant\mathbb{E}\bigg[\sum_{k=1}^{2^m}\frac{2g_a}{k+2}\bigg]\nonumber\\
		&\leqslant\mathbb{E}[F(\bar{x}_1)-F(x^*)]+\frac{LD^2}{2}\sum_{k=1}^{2^m}(\frac{2}{k+2})^2\nonumber\\
		&\quad+2D\sqrt{12L^2C^2_1+3C_3+12C_2}\sum_{k=1}^{2^m}\frac{2}{(k+2)^{1.5}}\nonumber\\
		&\leqslant4LD^2+2D\beta\sqrt{12L^2C^2_1+3C_3+12C_2}\nonumber\\
		&\quad+\mathbb{E}[F(\bar{x}_1)-F(x^*)].
	\end{align}
	Hence,
	\begin{align}
		\mathbb{E}[g_a]\leqslant&\frac{1}{m-1}\bigg(\mathbb{E}[F(\bar{x}_1)-F(x^*)]+4LD^2\nonumber\\
		&\quad+2D\beta\sqrt{12L^2C^2_1+3C_3+12C_2}\bigg)\nonumber\\
		&=\frac{1}{\mathrm{log}_2(K)-1}\bigg(\mathbb{E}[F(\bar{x}_1)-F(x^*)]+4LD^2\nonumber\\
		&\quad+2D\beta\sqrt{12L^2C^2_1+3C_3+12C_2}\bigg),
	\end{align}
	where the  equality is due to  $K=2^m$.
\end{proof}
\ifCLASSOPTIONcaptionsoff
  \newpage
\fi

\bibliographystyle{IEEEtran}
 \bibliography{Reference}

	\begin{IEEEbiography}[{\includegraphics[width=1in,height=1.25in,clip,keepaspectratio]{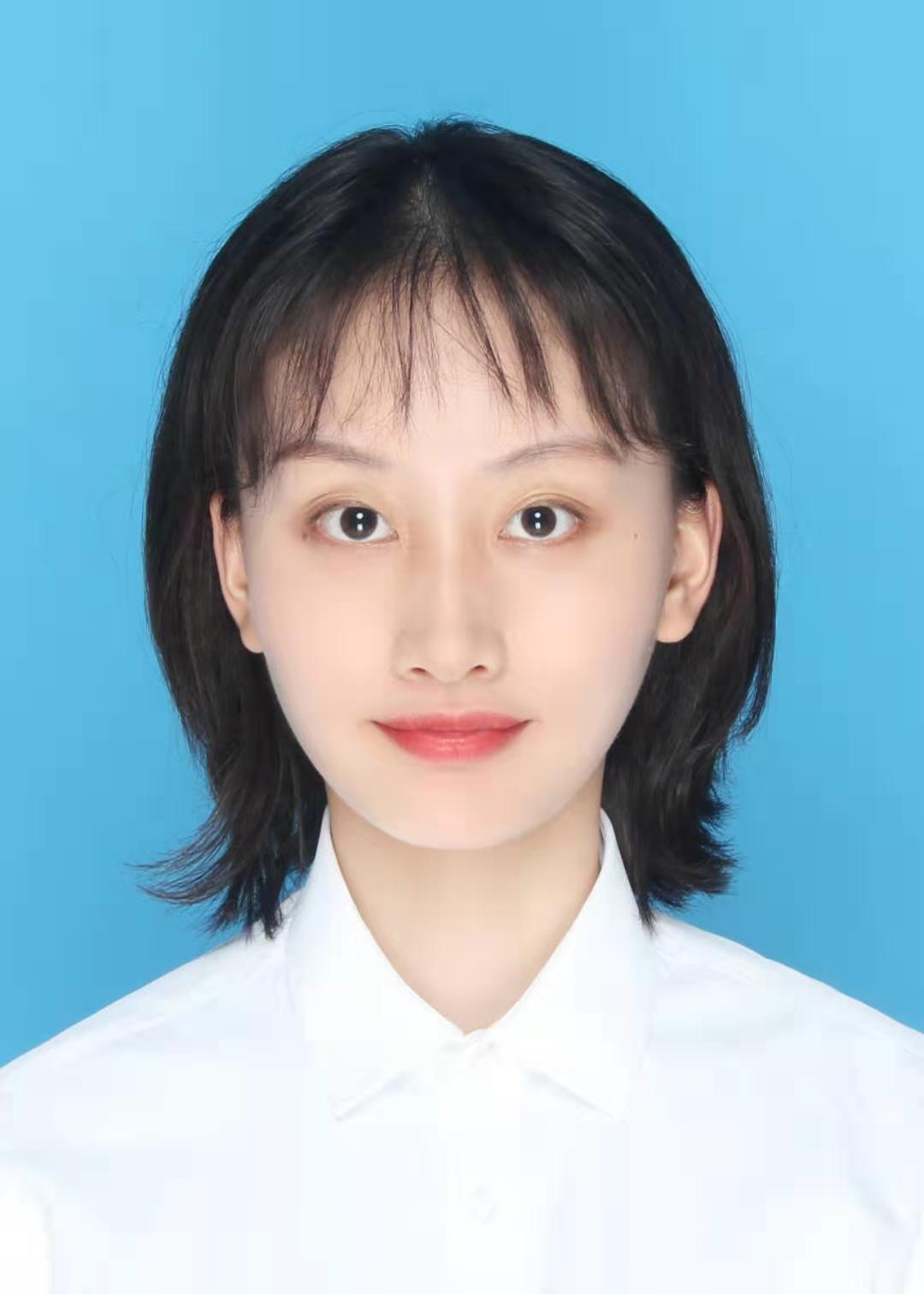}}]{Jie Hou} received the bachelor's degree in  computer science and technology from School of Computer and Information Engineering, Luoyang Institute of Science and Technology, China, in 2017, the master's degree in  computer science and technology from School of Computer Science and Software Engineering, Tiangong University, China, in 2020.
		She is currently pursuing the Ph.D. degree in control science and engineering with the School of Automation, Beijing Institute of Technology, Beijing, China.		
		Her current research interests include distributed optimization and stochastic optimization.
	\end{IEEEbiography}
	\begin{IEEEbiography}[{\includegraphics[width=1in,height=1.25in,clip,keepaspectratio]{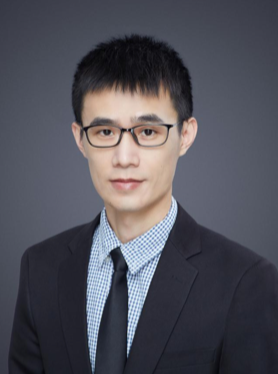}}]{Xianlin Zeng} (S'12-M'15) received the B.S. and M.S. degrees in Control Science and Engineering from the Harbin Institute of Technology, Harbin, China, in 2009 and 2011, respectively, and the Ph.D. degree in Mechanical Engineering from the Texas Tech University in 2015. He is currently an associate professor in the Key Laboratory of Intelligent Control and Decision of Complex Systems, School of Automation, Beijing Institute of Technology, Beijing, China. His current research interests include distributed optimization, distributed control, and distributed computation of network systems.	
	\end{IEEEbiography}
	\begin{IEEEbiography}[{\includegraphics[width=1in,height=1.5in,clip,keepaspectratio]{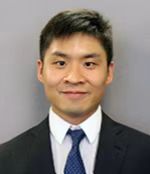}}] {Gang Wang} (M'18) received a B.Eng. degree in Automatic Control in 2011, and a Ph.D. degree in Control Science and Engineering in 2018, both from the Beijing Institute of Technology, Beijing, China. He also received a Ph.D. degree in Electrical and Computer Engineering from the University of Minnesota, Minneapolis, USA, in 2018, where he stayed as a postdoctoral researcher until July 2020. Since August 2020, he has been a professor with the School of Automation at the Beijing Institute of Technology.
		
		His research interests focus on the areas of signal processing, control, and reinforcement learning with applications to cyber-physical systems and multi-agent systems.
		He was the recipient of the Best Paper Award from the Frontiers of Information Technology \& Electronic Engineering (FITEE) in 2021, the Excellent Doctoral Dissertation Award from the Chinese Association of Automation in 2019, the Best Student Paper Award from the 2017 European Signal Processing Conference, and the Best Conference Paper at the 2019 IEEE Power \& Energy Society General Meeting. He is currently on the editorial board of Signal Processing, IEEE Open Journal of Control Systems, and IEEE Transactions on Signal and Information Processing over Networks. 
	\end{IEEEbiography}
	\begin{IEEEbiography}[{\includegraphics[width=1in,height=1.25in,clip,keepaspectratio]{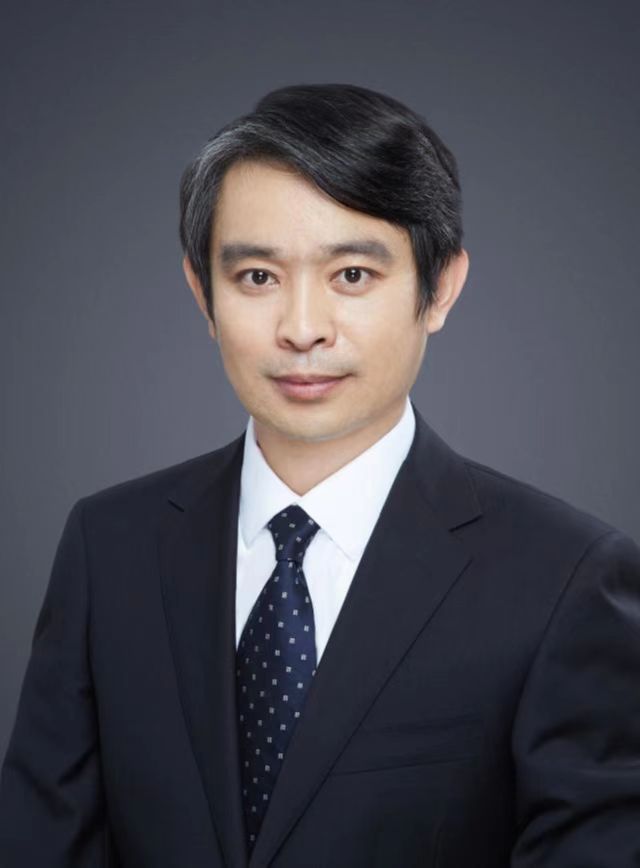}}]{Jian Sun} received his B.Sc. degree from the Department of Automation and Electric Engineering, Jilin Institute of Technology, Changchun, China, in 2001, the M.Sc. degree from the Changchun Institute of Optics, Fine Mechanics and Physics, Chinese Academy of Sciences (CAS), Changchun, China, in 2004, and the Ph.D. degree from the Institute of Automation, CAS, Beijing, China, in 2007.
		
		He was a Research Fellow with the Faculty of Advanced Technology, University of Glamorgan, Pontypridd, U.K., from 2008 to 2009. He was a Post-Doctoral Research Fellow with the Beijing Institute of Technology, Beijing, from 2007 to 2010. In 2010, he joined the School of Automation, Beijing Institute of Technology, where he has been a Professor since 2013. His current research interests include networked control systems, time-delay systems, and security of cyber-physical systems.
		
		Dr. Sun is an editorial board member of the IEEE Transactions on Systems, Man and Cybernetics: Systems, the Journal of Systems Science \& Complexity, and Acta. Automatica Sinica.
	\end{IEEEbiography}		

\begin{IEEEbiography}[{\includegraphics[width=1in,height=1.25in,clip,keepaspectratio]{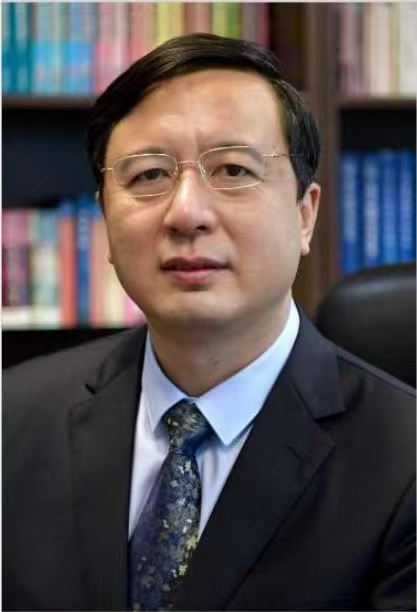}}]{Jie Chen} (M'09-SM'12-F'19) received his B.Sc., M.Sc., and the Ph.D. degrees in control theory and control engineering from the Beijing Institute of Technology, Beijing, China, in 1986, 1996, and 2001, respectively. From 1989 to 1990, he was a visiting scholar at the California State University, Long Beach, California, USA. From 1996 to 1997, he was a research fellow with the School of Engineering at the University of Birmingham, Birmingham, UK. 
			He is a Professor with the School of Automation, Beijing Institute of Technology, where he serves as the Director of the Key Laboratory of Intelligent Control and Decision of Complex Systems. He also serves as the President of Tongji University, Shanghai, China. 
			His research interests include multiagent systems, multiobjective optimization and decision, and constrained nonlinear control.
			
			Prof. Chen is currently the Editor-in-Chief of Unmanned Systems, Autonomous Intelligent Systems, and
		 Journal of Systems Science and Complexity. He has served on the editorial boards
			of several journals, including the
			IEEE Transactions on Cybernetics, International Journal of Robust
			and Nonlinear Control, and Science China Information Sciences. He is a Fellow of IEEE, IFAC, and a member of the Chinese Academy of Engineering.
			
			
			%

\end{IEEEbiography}
\vfill
\end{document}